\documentclass[a4paper]{scrartcl}

\usepackage[utf8]{inputenc}
\usepackage[T1]{fontenc}
\usepackage{lmodern}
\usepackage[svgnames,dvipsnames,rgb]{xcolor} %Colors
\usepackage{amsfonts}
\usepackage{amsmath} 
\usepackage{amssymb}
\usepackage{amsthm}
\usepackage{graphicx}
\usepackage{hyperref}
\hypersetup{linktocpage=true,colorlinks=true,linkcolor=RoyalBlue,citecolor=PineGreen,urlcolor=RoyalBlue}
\usepackage{float}
\usepackage{tikz,pgf}
\usetikzlibrary{calc}
\usetikzlibrary{patterns}
\usetikzlibrary{angles}
\usetikzlibrary{backgrounds}

\usepackage{cleveref}
\usepackage{subcaption}

\usepackage{enumitem}
\newcommand{\addPar}{\hyperref[def:Grec]{\textsc{Add4-cycle}}}
\newcommand{\closePar}{\hyperref[def:Grec]{\textsc{Close4-cycle}}}

\newtheorem{theorem}{Theorem}[section]
\newtheorem{lemma}[theorem]{Lemma}
\newtheorem{proposition}[theorem]{Proposition}
\newtheorem{corollary}[theorem]{Corollary}

\theoremstyle{definition}\newtheorem{definition}[theorem]{Definition}
\theoremstyle{definition}
\theoremstyle{definition}\newtheorem{remark}[theorem]{Remark}

\newcommand{\sage}{\textsc{SageMath}}
\newcommand{\FlexRiLoG}{\textsc{FlexRiLoG}}

\newcommand{\Grec}{\mathcal{G}_\text{rec}}
\newcommand{\parallelogramTilingFramework}{carpet framework}
\newcommand{\Pframework}{P-framework}

\newcommand{\oriented}[2]{(#1,#2)}

\newcommand{\eqwithreference}[1]{\stackrel{\text{\footnotesize\mbox{#1}}}{=}}

\newcommand{\cartProd}{\mathbin{\square}}

\newcommand{\blue}{\text{blue}}
\newcommand{\red}{\text{red}}

\newcommand{\RR}{\mathbb{R}}

\newcommand{\ci}{i}

\colorlet{ecol}{black!50!white}
\definecolor{colR}{rgb}{.932,.172,.172} %x11 Firebrick2
\definecolor{colB}{rgb}{.255,.41,.884} %svgnames RoyalBlue
\colorlet{colG}{Gold}
\colorlet{col1}{LightGreen}
\colorlet{col2}{IndianRed!80!white}
\colorlet{col3}{Gold}
\colorlet{col4}{LightSkyBlue}
\colorlet{col5}{BurlyWood}
\colorlet{col6}{Purple!60!white}
\colorlet{ncol}{DarkSeaGreen}

\tikzstyle{gvertex}=[circle, draw, fill=black, inner sep=0pt, minimum size=4pt]
\tikzstyle{smallgvertex}=[circle, draw, fill=black, inner sep=0pt, minimum size=2pt]
\tikzstyle{midgvertex}=[circle, draw, fill=black, inner sep=0pt, minimum size=3pt]

\colorlet{colvR}{black!70!white}
\tikzstyle{fvertex}=[circle,thick,draw=colvR,fill=white,inner sep=0pt, minimum size=4.5pt]
\tikzstyle{midfvertex}=[circle,thick,draw=colvR, fill=white, inner sep=0pt, minimum size=3.25pt]
\tikzstyle{smallfvertex}=[circle,thick,draw=colvR, fill=white, inner sep=0pt, minimum size=2pt]

\tikzstyle{edge}=[line width=1.5pt,ecol]
\tikzstyle{redge}=[edge,colR]
\tikzstyle{bedge}=[edge,colB]
\tikzstyle{gedge}=[edge,colG]

\tikzstyle{gridl}=[ecol]
\tikzstyle{gridp}=[inner sep=1pt,circle,fill=black!70!white]
\tikzstyle{sym}=[ecol,dashed]

\tikzstyle{axes}=[gridl,-latex]

\tikzstyle{ngvertex}=[gvertex, draw=ncol, fill=ncol]
\tikzstyle{edgeq}=[edge,gray!60,densely dashed]
\tikzstyle{nedge}=[edge,ncol]
\tikzstyle{bdedge}=[line width=1.5pt,colB, densely dashed]
\tikzstyle{rdedge}=[line width=1.5pt,colR, densely dashed]

\tikzstyle{genericgraph}=[dashed]
\tikzstyle{labelsty}=[font=\scriptsize]
\tikzstyle{indicatededge}=[pin={[pin distance=6pt,pin edge={thin}]20:},pin={[pin distance=6pt,pin edge={thin}]-20:},pin={[pin distance=10pt,pin edge={thin}]2:}]

\tikzstyle{ribbon}=[densely dashed,rounded corners,line width=1pt,shorten <= -6pt, shorten >= -6pt]
\tikzstyle{brace}=[line width=1pt,ecol]

\tikzstyle{category}=[font=\small]
\tikzstyle{interiorangle}=[angle radius=0.3cm]

\title{Bracing frameworks \\consisting of parallelograms}
\author{Georg Grasegger\thanks{Johann Radon Institute for Computational and Applied Mathematics (RICAM), Austrian Academy of Sciences}
\and Jan Legerský\thanks{Johannes Kepler University Linz, Research Institute for Symbolic Computation (RISC)}
\thanks{Department of Applied Mathematics, Faculty of Information Technology, Czech Technical University in Prague}}
\date{}

\begin{document}

\maketitle

\begin{abstract}
	A rectangle in the plane can be continuously deformed preserving its edge lengths,
	but adding a diagonal brace prevents such a deformation.   
	Bolker and Crapo characterized combinatorially which choices of braces make a grid of squares 
	infinitesimally rigid using a \emph{bracing graph}: 
	a bipartite graph whose vertices are the columns and rows of the grid, 
	and a row and column are adjacent if and only if they meet at a braced square.
	Duarte and Francis generalized the notion of the bracing graph to rhombic carpets,
	proved that the connectivity of the bracing graph implies rigidity and stated the other implication without proof.
	Nagy Kem gives the equivalence in the infinitesimal setting. 
	We consider continuous deformations of braced frameworks consisting of a graph from a more general class
	and its placement in the plane such that every 4-cycle forms a parallelogram.
	We show that rigidity of such a braced framework is equivalent to the non-existence of a special edge coloring,
	which is in turn equivalent to the corresponding bracing graph being connected.
\end{abstract}

\section{Introduction}
A planar framework is a graph together with a placements of its vertices in the plane.
If there is a non-trivial flex (a deformation of the placement preserving the distances between adjacent vertices that
is not induced by a rigid motion), then the framework is said to be flexible, otherwise rigid.
Bolker and Crapo~\cite{BolkerCrapo} studied infinitesimal flexibility of a framework corresponding to a grid of squares
with some squares being braced by adding diagonals, see \Cref{fig:introgrid}.
They construct a bipartite graph by taking the columns and rows of the grid to be the two parts of the vertex set;
a column and row are connected if and only if their common square is braced.
They showed that a braced grid is infinitesimally rigid, i.e., has no non-trivial first order flex,
if and only if the bipartite graph is connected. 

\begin{figure}[ht]
  \centering
  \begin{tikzpicture}[scale=0.75]
    \foreach \x [remember=\x as \xr (initially 1)] in {1,2,3,4}
    {
			\foreach \y [remember=\y as \yr (initially 1)] in {1,2,3,4}
			{
				\node[fvertex] (v\x\y) at (\x,\y) {};
				\draw[edge] (v\x\yr)--(v\x\y);
				\draw[edge] (v\x\y)--(v\xr\y);
			}			
    }
  \end{tikzpicture}
  \qquad
  \begin{tikzpicture}[scale=0.75]
				\coordinate (xl) at (1,0);
				\coordinate (yl) at (0,1);
				\coordinate (o) at (0,0);
				\coordinate[rotate around=0:(o)] (a1) at ($(o)+(xl)$);
				\coordinate[rotate around=30:(a1)] (a2) at ($(a1)+(xl)$);
				\coordinate[rotate around=-20:(a2)] (a3) at ($(a2)+(xl)$);
				
				\coordinate[rotate around=0:(o)] (b1) at ($(o)+(yl)$);
				\coordinate[rotate around=20:(b1)] (b2) at ($(b1)+(yl)$);
				\coordinate[rotate around=-20:(b2)] (b3) at ($(b2)+(yl)$);
				
				\foreach \a in {(o),(a1),(a2),(a3)}
				{
					\begin{scope}[shift={\a}]
							\draw[edge,rotate=90] (0,0) -- (1,0);
							\begin{scope}[shift={($(b1)-(o)$)},rotate=20]
								\begin{scope}[shift={($(b2)-(b1)$)},rotate=-40]
									\draw[edge,rotate=90] (0,0) -- (1,0);
								\end{scope}
								\draw[edge,rotate=90] (0,0) -- (1,0);
							\end{scope}
						\end{scope}
				}
				\foreach \b in {(o),(b1),(b2),(b3)}
				{
					  \begin{scope}[shift={\b}]
							\draw[edge] (0,0) -- (1,0);
							\begin{scope}[shift={($(a1)-(o)$)},rotate=30]
								\begin{scope}[shift={($(a2)-(a1)$)},rotate=-50]
									\draw[edge] (0,0) -- (1,0);
								\end{scope}
								\draw[edge] (0,0) -- (1,0);
							\end{scope}
						\end{scope}
				}
				\foreach \a in {o,a1,a2,a3}
				{
					\foreach \b in {o,b1,b2,b3}
					{
						\node[fvertex] at ($(\a) +(\b)$) {};
					}
				}
    \end{tikzpicture}
    \qquad
    \begin{tikzpicture}[scale=0.75]
			\foreach \x [remember=\x as \xr (initially 1)] in {1,2,3,4}
			{
				\foreach \y [remember=\y as \yr (initially 1)] in {1,2,3,4}
				{
					\node[fvertex] (v\x\y) at (\x,\y) {};
					\draw[edge] (v\x\yr)--(v\x\y);
					\draw[edge] (v\x\y)--(v\xr\y);
				}			
			}
			\draw[edge] (v13)--(v24);
			\draw[edge] (v23)--(v34);
			\draw[edge] (v33)--(v44);
			\draw[edge] (v22)--(v33);
			\draw[edge] (v21)--(v32);
		\end{tikzpicture}
		\qquad
		\begin{tikzpicture}[scale=0.75]
				\coordinate (o) at (0,0);
				\coordinate[rotate around=0:(o)] (a1) at ($(o)+(1,0)$);
				\coordinate[rotate around=40:(a1)] (a2) at ($(a1)+(1,0)$);
				\coordinate[rotate around=0:(a2)] (a3) at ($(a2)+(1,0)$);
				
				\coordinate[rotate around=0:(o)] (b1) at ($(o)+(0,1)$);
				\coordinate[rotate around=40:(b1)] (b2) at ($(b1)+(0,1)$);
				\coordinate[rotate around=0:(b2)] (b3) at ($(b2)+(0,1)$);
				
				\draw[edge,rotate=45] (0,0) -- (1.41421,0);
				\draw[edge,rotate=45,shift={(a2)}] (0,0) -- (1.41421,0);
				\draw[edge,rotate=45,shift={(b2)}] (0,0) -- (1.41421,0);
				\draw[edge,rotate=45,shift={($(a2)+(b2)$)}] (0,0) -- (1.41421,0);
				\draw[edge,rotate=85,shift={($(a1)+(b1)$)}] (0,0) -- (1.41421,0);
				
				\foreach \a in {(o),(a1),(a2),(a3)}
				{
					\begin{scope}[shift={\a}]
							\draw[edge,rotate=90] (0,0) -- (1,0);
							\begin{scope}[shift={($(b1)-(o)$)},rotate=40]
								\begin{scope}[shift={($(b2)-(b1)$)},rotate=-40]
									\draw[edge,rotate=90] (0,0) -- (1,0);
								\end{scope}
								\draw[edge,rotate=90] (0,0) -- (1,0);
							\end{scope}
						\end{scope}
				}
				\foreach \b in {(o),(b1),(b2),(b3)}
				{
					  \begin{scope}[shift={\b}]
							\draw[edge] (0,0) -- (1,0);
							\begin{scope}[shift={($(a1)-(o)$)},rotate=40]
								\begin{scope}[shift={($(a2)-(a1)$)},rotate=-40]
									\draw[edge] (0,0) -- (1,0);
								\end{scope}
								\draw[edge] (0,0) -- (1,0);
							\end{scope}
						\end{scope}
				}
				\foreach \a in {o,a1,a2,a3}
				{
					\foreach \b in {o,b1,b2,b3}
					{
						\node[fvertex] at ($(\a) +(\b)$) {};
					}
				}
    \end{tikzpicture}
    \caption{Grid frameworks can be deformed in a way that preserves the edge lengths.
    By bracing it (i.e.\ by adding diagonal edges) we can reduce the number of degrees of freedom.
    The left braced grid is rigid whereas the right one allows a flex.}
    \label{fig:introgrid}
\end{figure}
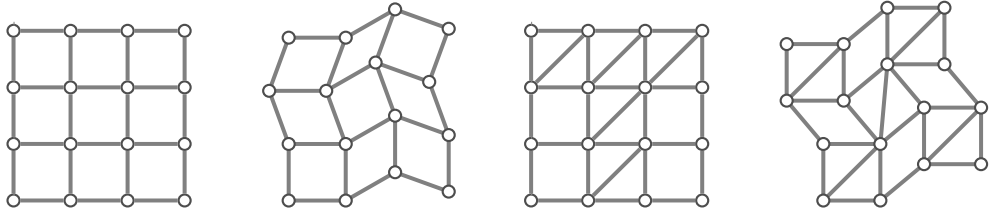

Generalizations to rectangular grids with holes~\cite{Gaspar1999,RecskiMatroid} or placing longer diagonals~\cite{Gaspar1998}
than those for a single 4-cycle have been studied as well as bracing by cables~\cite{RadicsRecski}.
Grids with rectilinear boundary are discussed in~\cite{Ito2015}.
Extensions to cubic grids have been studied~\cite{Bolker,Recski}.
The papers \cite{Ellenbroek2011,Zhang2015} describe the number of randomly added braces
for which the transition from rigid to flexible occurs.
A related problem to the rigidity of a grid is the rigidity of one- and multi-story building~\cite{Radics,RecskiStatics},
also with cables~\cite{RecskiTensegrities1,RecskiTensegrities2,RecskiTensegrities3}.
Simple forms of bracing grids are also known to be suitable as a puzzle, for science communication and for student's exercises (see for instance \cite{ServatiusGrid}).

In this work, we focus on parallelograms instead of squares,
and we allow a richer combinatorial structure than grids.
Flexibility of rhombic/parallelogramic tilings is studied by physicists
due to its relation with quasicrystals~\cite{Zhou2019}.
The bracing of rhombic carpets, which are 1-skeleta of finite simply
connected pieces of rhombic tilings, was investigated by Wester~\cite{Wester}.
Duarte and Francis~\cite{DuarteFrancis} formalized the notions necessary to study the flexibility of rhombic carpets:
a natural step from columns and rows of a grid towards a rhombic carpet is to take \emph{ribbons}.
These are sequences of rhombi such that
every two consecutive ones share an edge and all these edges are parallel.
Following the idea of Bolker and Crapo, Duarte and Francis construct a \emph{bracing graph}
whose vertices are the ribbons
and two ribbons are adjacent if they have a common rhombus that is braced.
They prove that if the constructed graph is connected, than the braced rhombic carpet is rigid.
Further, they state the other implication without proof. 
We thank Eliana Duarte for pointing out this statement to us and
sharing some hints about a possible proof~\cite{Duarte}.
Nagy Kem~\cite{NagyKem2017} translates the infinitesimal
rigidity of a braced rhombic carpet to the rigidity of an auxiliary framework
which in turn corresponds to the connectivity of the bracing graph.  

We formulate the problem of the flexibility of braced structures in terms of frameworks.
In particular, we define ribbons as equivalence classes on edges of the underlying graph using its 4-cycles.
We consider a special class of graphs, which we call \emph{ribbon-cutting} graphs.
A connected graph is ribbon-cutting if every ribbon is an edge cut,
i.e., removing the edges of the ribbon makes the graph disconnected.
Regarding the placement, we ask all 4-cycles to form parallelograms, see \Cref{fig:introcarpet}.
Notice that the frameworks we consider --- we call them \Pframework{}s --- 
form a proper superset of the frameworks corresponding to rhombic carpets and rectangular grids (without holes).
The question we address is analogous to the one by Bolker, Crapo, Duarte and others,
namely, characterization of choices of braces of parallelograms yielding flexible/rigid \Pframework{}s.
Contrary to Bolker, Crapo and Nagy Kem, we consider finite flexes, not infinitesimal ones.
Furthermore, we use a recently established method of special edge colorings to prove our results.

\begin{figure}[ht]
  \centering
  \begin{tikzpicture}[rotate=-90]
    \node[fvertex] (0) at (0,0) {};
    \node[fvertex] (1) at (0,1) {};
    \node[fvertex,rotate around=0:(0)] (2) at ($(0)+(1.00,0)$) {};
    \node[fvertex] (3) at ($(2)+(1)-(0)$) {};
    \node[fvertex,rotate around=45:(2)] (4) at ($(2)+(0.50,0)$) {};
    \node[fvertex] (5) at ($(4)+(3)-(2)$) {};
    \node[fvertex,rotate around=120:(1)] (6) at ($(1)+(0.50,0)$) {};
    \node[fvertex] (7) at ($(6)+(3)-(1)$) {};
    \node[fvertex,rotate around=72:(7)] (8) at ($(7)+(1.50,0)$) {};
    \node[fvertex] (9) at ($(8)+(3)-(7)$) {};
    \node[fvertex] (10) at ($(5)+(9)-(3)$) {};
    \node[fvertex] (11) at ($(8)+(10)-(9)$) {};
    \node[fvertex] (12) at ($(6)+(8)-(7)$) {};
    \node[fvertex] (13) at ($(0)+(6)-(1)$) {};
    \node[fvertex,rotate around=144:(6)] (14) at ($(6)+(1.00,0)$) {};
    \node[fvertex] (15) at ($(14)+(13)-(6)$) {};
    \node[fvertex,rotate around=108:(6)] (16) at ($(6)+(1.00,0)$) {};
    \node[fvertex] (17) at ($(16)+(14)-(6)$) {};
    \node[fvertex] (18) at ($(12)+(16)-(6)$) {};
    \node[fvertex] (19) at ($(17)+(18)-(16)$) {};
    \node[fvertex,rotate around=90:(12)] (20) at ($(12)+(0.75,0)$) {};
    \node[fvertex] (21) at ($(20)+(8)-(12)$) {};
    \node[fvertex] (22) at ($(11)+(21)-(8)$) {};
    \node[fvertex] (23) at ($(18)+(20)-(12)$) {};
    \node[fvertex] (24) at ($(23)+(21)-(20)$) {};
    \node[fvertex] (25) at ($(19)+(23)-(18)$) {};
    \node[fvertex] (26) at ($(24)+(22)-(21)$) {};
    \node[fvertex] (27) at ($(25)+(24)-(23)$) {};
    \node[fvertex] (28) at ($(27)+(26)-(24)$) {};
    \node[fvertex] (29) at ($(25)+(28)-(27)$) {};
    \draw[edge] (0)--(1)
        (0)--(2) (2)--(3) (3)--(1)
        (2)--(4) (4)--(5) (5)--(3)
        (1)--(6) (6)--(7) (7)--(3)
        (7)--(8) (8)--(9) (9)--(3)
        (5)--(10) (10)--(9) 
        (8)--(11) (11)--(10) 
        (6)--(12) (12)--(8) 
        (0)--(13) (13)--(6) 
        (6)--(14) (14)--(15) (15)--(13)
        (6)--(16) (16)--(17) (17)--(14)
        (12)--(18) (18)--(16) 
        (17)--(19) (19)--(18) 
        (12)--(20) (20)--(21) (21)--(8)
        (11)--(22) (22)--(21) 
        (18)--(23) (23)--(20) 
        (23)--(24) (24)--(21) 
        (19)--(25) (25)--(23) 
        (24)--(26) (26)--(22) 
        (25)--(27) (27)--(24) 
        (27)--(28) (28)--(26) 
        (25)--(29) (29)--(28);
	\end{tikzpicture}
	\qquad
	\begin{tikzpicture}[rotate=-90]
    \node[fvertex] (0) at (0,0) {};
    \node[fvertex] (1) at (0,1) {};
    \node[fvertex,rotate around=0:(0)] (2) at ($(0)+(1.00,0)$) {};
    \node[fvertex] (3) at ($(2)+(1)-(0)$) {};
    \node[fvertex,rotate around=45:(2)] (4) at ($(2)+(0.50,0)$) {};
    \node[fvertex] (5) at ($(4)+(3)-(2)$) {};
    \node[fvertex,rotate around=110:(1)] (6) at ($(1)+(0.50,0)$) {};
    \node[fvertex] (7) at ($(6)+(3)-(1)$) {};
    \node[fvertex,rotate around=80:(7)] (8) at ($(7)+(1.50,0)$) {};
    \node[fvertex] (9) at ($(8)+(3)-(7)$) {};
    \node[fvertex] (10) at ($(5)+(9)-(3)$) {};
    \node[fvertex] (11) at ($(8)+(10)-(9)$) {};
    \node[fvertex] (12) at ($(6)+(8)-(7)$) {};
    \node[fvertex] (13) at ($(0)+(6)-(1)$) {};
    \node[fvertex,rotate around=130:(6)] (14) at ($(6)+(1.00,0)$) {};
    \node[fvertex] (15) at ($(14)+(13)-(6)$) {};
    \node[fvertex,rotate around=100:(6)] (16) at ($(6)+(1.00,0)$) {};
    \node[fvertex] (17) at ($(16)+(14)-(6)$) {};
    \node[fvertex] (18) at ($(12)+(16)-(6)$) {};
    \node[fvertex] (19) at ($(17)+(18)-(16)$) {};
    \node[fvertex,rotate around=80:(12)] (20) at ($(12)+(0.75,0)$) {};
    \node[fvertex] (21) at ($(20)+(8)-(12)$) {};
    \node[fvertex] (22) at ($(11)+(21)-(8)$) {};
    \node[fvertex] (23) at ($(18)+(20)-(12)$) {};
    \node[fvertex] (24) at ($(23)+(21)-(20)$) {};
    \node[fvertex] (25) at ($(19)+(23)-(18)$) {};
    \node[fvertex] (26) at ($(24)+(22)-(21)$) {};
    \node[fvertex] (27) at ($(25)+(24)-(23)$) {};
    \node[fvertex] (28) at ($(27)+(26)-(24)$) {};
    \node[fvertex] (29) at ($(25)+(28)-(27)$) {};
    \draw[edge] (0)--(1)
        (0)--(2) (2)--(3) (3)--(1)
        (2)--(4) (4)--(5) (5)--(3)
        (1)--(6) (6)--(7) (7)--(3)
        (7)--(8) (8)--(9) (9)--(3)
        (5)--(10) (10)--(9) 
        (8)--(11) (11)--(10) 
        (6)--(12) (12)--(8) 
        (0)--(13) (13)--(6) 
        (6)--(14) (14)--(15) (15)--(13)
        (6)--(16) (16)--(17) (17)--(14)
        (12)--(18) (18)--(16) 
        (17)--(19) (19)--(18) 
        (12)--(20) (20)--(21) (21)--(8)
        (11)--(22) (22)--(21) 
        (18)--(23) (23)--(20) 
        (23)--(24) (24)--(21) 
        (19)--(25) (25)--(23) 
        (24)--(26) (26)--(22) 
        (25)--(27) (27)--(24) 
        (27)--(28) (28)--(26) 
        (25)--(29) (29)--(28);
	\end{tikzpicture}
	\caption{Carpet frameworks can be deformed in a way that preserves the edge lengths.}
	\label{fig:introcarpet}
\end{figure}
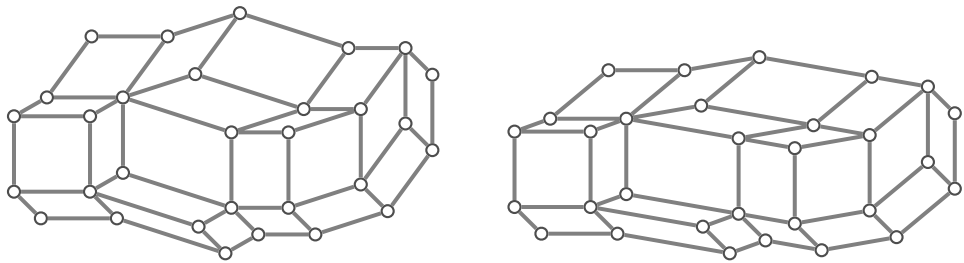

The notion of \emph{NAC-colorings} was
developed in our previous paper~\cite{flexibleLabelings}.
A NAC-coloring is a surjective edge coloring of a graph by red and blue
such that for every cycle of the graph,
either all edges have the same color or there are at least two edges of each color.
We proved that a graph has a flexible framework if and only if it has a NAC-coloring.
It appears that the techniques used to prove the theorem fit nicely to the context
of bracing P-frameworks if we restrict ourselves to certain NAC-colorings:
a NAC-coloring is called \emph{cartesian} if there are no two vertices connected by a red path and blue path simultaneously.
The non-existence of a cartesian NAC-coloring serves as a bridge in the proof that  
a braced \Pframework{} is rigid if and only
if the corresponding bracing graph (defined analogously to \cite{DuarteFrancis}) is connected.
Our results can be summarized as follows.
 \begin{theorem}
 	\label{thm:main_result}
	For a braced \Pframework{} $(G,\rho)$,
	the following statements are equivalent:
	\begin{enumerate}
	  \item $(G,\rho)$ is rigid,
	  \item $G$ has no cartesian NAC-coloring, and
	  \item the bracing graph of $G$ is connected. 
	\end{enumerate}
	In particular, the minimum number of braces making a framework rigid
	is one less than the number of ribbons of its underlying graph. 
\end{theorem}

We implement the concepts introduced in this paper by extending our \sage{} package \FlexRiLoG~\cite{flexrilog}.
We encourage the reader to experiment with the Jupyter notebook available on-line 
on \url{https://jan.legersky.cz/bracingFrameworks}.

The paper is organized as follows: \Cref{sec:prelim} recalls the notions from Rigidity theory and NAC-colorings.
The name \emph{cartesian} NAC-coloring is justified.
We define ribbons, parallelogram placements and \Pframework{}s in \Cref{sec:ribbons}.
Furthermore, ribbon-cutting graphs are defined in that section and we prove some results needed later in the paper.
We also construct recursively a subset of ribbon-cutting graphs and show that each graph
in this class has a parallelogram placement.
The class contains the underlying graphs of all frameworks corresponding to a slight generalization of rhombic carpets
(allowing parallelograms instead of rhombi).
Finally, we formalize bracing and the notion of bracing graph in our context.
\Cref{sec:flex} provides the proofs yielding \Cref{thm:main_result}.

\section{Preliminaries}\label{sec:prelim}
In this section we present basic notation and definitions.
The ideas are based on previous work using special edge colorings to find flexes of graphs.
We introduce these colorings here and describe what we mean by flexibility.

\begin{definition}
	Let $G=(V_G,E_G)$ be a connected graph.
	A map $\rho : V_G \rightarrow \RR^2$ such that $\rho(u) \neq \rho(v)$ for all edges $uv \in E_G$
	is a \emph{placement}.
	The pair $(G,\rho)$ is called a \emph{framework}.
\end{definition}
	
\begin{definition}
	Two frameworks $(G,\rho)$ and $(G,\rho')$ are \emph{equivalent} if 
	\begin{align*}
		\| \rho(u) - \rho(v) \| = \| \rho'(u) - \rho'(v)\|
	\end{align*}
	for all $uv \in E_G$.
	Two placements $\rho$ and $\rho'$ are \emph{congruent} if there exists a Euclidean
	isometry~$M$ of $\RR^2$ such that $M \rho'(v) = \rho(v)$ for all $v \in V_G$.
\end{definition}

\begin{definition}
	A \emph{flex} of the framework $(G,\rho)$ is a continuous path $t \mapsto \rho_t$, $t \in [0,1]$,
	in the space of placements of $G$ such that $\rho_0= \rho$ and each $(G,\rho_t)$ is equivalent to~$(G,\rho)$.
	The flex is called trivial if $\rho_t$ is congruent to $\rho$ for all $t \in [0,1]$.
	
	We define a framework to be \emph{(proper) flexible}
	if there is a non-trivial flex in $\RR^2$ (with injective placements).
	Otherwise it is called \emph{rigid}.
\end{definition}

In a previous paper we classify the graphs that have flexible frameworks by a special edge coloring, which is called NAC-coloring.
\begin{definition}
	Let~$G$ be a graph. A coloring of edges $\delta\colon  E_G\rightarrow \{\text{\blue{}, \red{}}\}$ 
	is called a \emph{NAC-coloring},
	if it is surjective and for every cycle in $G$,
	either all edges have the same color, or
	there are at least 2 edges in each color (see \Cref{fig:nac}).
	The NAC-coloring $\delta$ gives subgraphs
	\begin{align*}
		G_\red^\delta = (V_G, \{e\in E_G \colon \delta(e) = \red\})\, \text{ and }
		G_\blue^\delta = (V_G, \{e\in E_G \colon \delta(e) = \blue\})\,.
	\end{align*} 
\end{definition}

\begin{theorem}[\cite{flexibleLabelings}]
	\label{thm:nacflexible}
	A connected non-trivial graph allows a flexible framework if and only if it has a NAC-coloring.
\end{theorem}
Two non-adjacent vertices $u$ and $v$ overlap in the flex constructed in the proof of the theorem in \cite{flexibleLabelings}
if and only if there is a \red{} path from $u$ to $v$ and a \blue{} path from $u$ to~$v$.
In order to avoid overlapping vertices, we focus on a special type of NAC-colorings.
\begin{definition}
	A NAC-coloring $\delta$ of a graph $G$ is called \emph{cartesian}
	if no two distinct vertices are connected by a red and blue path simultaneously.
\end{definition}
	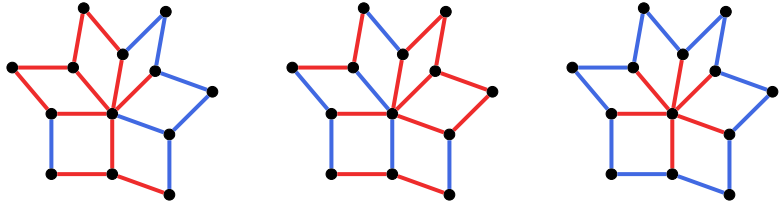
\begin{figure}[ht]
	  \centering
	  \begin{tikzpicture}[scale=0.8]
	    \node[gvertex] (a) at (0,0) {};
	    \node[gvertex] (b) at (1,0) {};
	    \node[gvertex,rotate around=40:(b)] (c) at ($(b)+(0,1)$) {};
	    \node[gvertex] (d) at ($(a)+(c)-(b)$) {};
	    \node[gvertex,rotate around=-10:(b)] (e) at ($(b)+(0,1)$) {};
	    \node[gvertex] (f) at ($(c)+(e)-(b)$) {};
	    \node[gvertex,rotate around=-45:(b)] (g) at ($(b)+(0,1)$) {};
	    \node[gvertex] (h) at ($(e)+(g)-(b)$) {};
	    \node[gvertex,rotate around=-110:(b)] (i) at ($(b)+(0,1)$) {};
	    \node[gvertex] (j) at ($(g)+(i)-(b)$) {};
	    \node[gvertex,rotate around=-180:(b)] (k) at ($(b)+(0,1)$) {};
	    \node[gvertex] (l) at ($(i)+(k)-(b)$) {};
	    \node[gvertex] (m) at ($(a)+(k)-(b)$) {};
	    \draw[redge] (a)--(b) (b)--(c) (c)--(d) (d)--(a) (b)--(e) (k)--(l) (b)--(k) (m)--(k) (e)--(f) (f)--(c) (b)--(g); 
	    \draw[bedge] (g)--(h) (h)--(e) (b)--(i) (i)--(j) (j)--(g) (l)--(i) (a)--(m);
	  \end{tikzpicture}
	  \qquad
	  \begin{tikzpicture}[scale=0.8]
	    \node[gvertex] (a) at (0,0) {};
	    \node[gvertex] (b) at (1,0) {};
	    \node[gvertex,rotate around=40:(b)] (c) at ($(b)+(0,1)$) {};
	    \node[gvertex] (d) at ($(a)+(c)-(b)$) {};
	    \node[gvertex,rotate around=-10:(b)] (e) at ($(b)+(0,1)$) {};
	    \node[gvertex] (f) at ($(c)+(e)-(b)$) {};
	    \node[gvertex,rotate around=-45:(b)] (g) at ($(b)+(0,1)$) {};
	    \node[gvertex] (h) at ($(e)+(g)-(b)$) {};
	    \node[gvertex,rotate around=-110:(b)] (i) at ($(b)+(0,1)$) {};
	    \node[gvertex] (j) at ($(g)+(i)-(b)$) {};
	    \node[gvertex,rotate around=-180:(b)] (k) at ($(b)+(0,1)$) {};
	    \node[gvertex] (l) at ($(i)+(k)-(b)$) {};
	    \node[gvertex] (m) at ($(a)+(k)-(b)$) {};
	    \draw[redge] (a)--(b)  (c)--(d) (b)--(e) (f)--(c) 
	    	(k)--(l) (b)--(g) (g)--(h) (h)--(e) (m)--(k) (b)--(i) (i)--(j) (j)--(g); 
	    \draw[bedge] (b)--(c) (d)--(a)  (e)--(f) (b)--(k) (a)--(m) (l)--(i);
	  \end{tikzpicture}
	  \qquad
	  \begin{tikzpicture}[scale=0.8]
	    \node[gvertex] (a) at (0,0) {};
	    \node[gvertex] (b) at (1,0) {};
	    \node[gvertex,rotate around=40:(b)] (c) at ($(b)+(0,1)$) {};
	    \node[gvertex] (d) at ($(a)+(c)-(b)$) {};
	    \node[gvertex,rotate around=-10:(b)] (e) at ($(b)+(0,1)$) {};
	    \node[gvertex] (f) at ($(c)+(e)-(b)$) {};
	    \node[gvertex,rotate around=-45:(b)] (g) at ($(b)+(0,1)$) {};
	    \node[gvertex] (h) at ($(e)+(g)-(b)$) {};
	    \node[gvertex,rotate around=-110:(b)] (i) at ($(b)+(0,1)$) {};
	    \node[gvertex] (j) at ($(g)+(i)-(b)$) {};
	    \node[gvertex,rotate around=-180:(b)] (k) at ($(b)+(0,1)$) {};
	    \node[gvertex] (l) at ($(i)+(k)-(b)$) {};
	    \node[gvertex] (m) at ($(a)+(k)-(b)$) {};
	    \draw[redge] (a)--(b) (b)--(e) (b)--(c) (b)--(g)  (b)--(i) (b)--(k); 
	    \draw[bedge]  (c)--(d) (d)--(a) (f)--(c) (k)--(l) (g)--(h)
	    	 (i)--(j) (j)--(g)(h)--(e) (m)--(k) (e)--(f) (a)--(m) (l)--(i);
	  \end{tikzpicture}
	  \caption{A coloring that is not NAC-coloring (left),  cartesian NAC-coloring (middle)
	  	and non-cartesian NAC-coloring (right).}
	  \label{fig:nac}
	\end{figure}
\begin{remark}
	\label{rem:cartesianNACcharacterization}
	A NAC-coloring $\delta$ of a graph $G$ is cartesian if and only if
	for every connected component $R$ of $G_\red^\delta$ and $B$ of $G_\blue^\delta$,
	the intersection of the vertex sets of $R$ and $B$ contains at most one vertex.
\end{remark}
Notice that in a cartesian NAC-coloring, a 4-cycle subgraph is monochromatic,
or the opposite edges have the same color. 

Recall that the cartesian product of graphs $G$ and $H$ is given by
\[
	G\cartProd H = (V_G \times V_H, \{(u,u')(v,v') \colon (u=v \land u'v' \in E_H) \lor (u'=v' \land uv \in E_G)\})\,.
\]
By coloring edges coming from $G$ by \red{} and the rest \blue{}, the following holds.
\begin{theorem}[\cite{Biswas,Hammack}]\label{thm:cartnac}
	The cartesian product of any two nontrivial graphs $G$ and $H$ has a cartesian NAC-coloring.
\end{theorem}
We remark that the statement of \Cref{thm:cartnac} has been pointed out by~\cite{Hammack} independently of~\cite{Biswas}.
In \cite{Biswas} a cartesian NAC-coloring is called \emph{good} 
since applying the grid construction described in~\cite{flexibleLabelings} 
yields a proper flexible framework, whereas for a non-cartesian NAC-coloring there are overlapping vertices. 
Our naming is motivated by the fact that the converse statement
can be proved using ideas from \cite{HammackImrichKlavzar} about embeddings of graphs into cartesian products.
\begin{theorem}
	\label{thm:cartNACimpliesSubgraph}
	If a graph $G$ has a cartesian NAC-coloring, then there are graphs $Q_1,Q_2$ with at least two vertices each 
	and an injective graph morphism $h:G \rightarrow Q_1\cartProd Q_2$ such that each vertex in $V_{Q_1}\cup V_{Q_2}$
	occurs as a coordinate of a vertex in $h(G)$. In particular, $G$ can be viewed as a subgraph of $Q_1 \cartProd Q_2$.
\end{theorem}
\begin{proof}
	Let $\delta$ be a cartesian NAC-coloring of $G$.
	Let $R_1, \dots, R_m$, resp.\ $B_1, \dots, B_n$, be the vertex sets of the connected components of $G_\red^\delta$, resp.\ $G_\blue^\delta$.
	Since $\delta$ is surjective and no \blue{} edge can connect vertices of the same \red{} component \cite[Lemma~2.4]{flexibleLabelings},
	$m\geq 2$ and $n\geq 2$.
	Let $\pi_\red : V_G \rightarrow \{R_1, \dots, R_m\}$ and $\pi_\blue : V_G \rightarrow \{B_1, \dots, B_m\}$
	map a vertex to the vertex set of its \red{}, resp.\ \blue, component, namely, $\pi_\red(v) = R_i$ and $\pi_\blue(v) = B_j$ if $v\in R_i \cap B_j$. 
	We define the following quotient graphs
	\begin{align*}
		Q_1 &= \left(\{R_1, \dots, R_m\}, \{\pi_\red(u)\pi_\red(v) \colon uv \in E_G \text{ and } \delta(uv)=\blue\}\right)\,, \\
		Q_2 &= \left(\{B_1, \dots, B_n\}, \{\pi_\blue(u)\pi_\blue(v) \colon uv \in E_G \text{ and } \delta(uv)=\red\}\right)\,.
	\end{align*}
	Let $Q$ be the cartesian product of $Q_1$ and $Q_2$, and $h:V_G\rightarrow V_Q$ be the graph morphism given by
	\begin{equation*}
		h(v) = (\pi_\red(v), \pi_\blue(v))\,.
	\end{equation*}
	We check that it is indeed a morphism: if $uv$ is an edge of $G$, w.l.o.g. \red{},
	then $\pi_\red(u)=\pi_\red(v)$ and $\pi_\blue(u)\neq\pi_\blue(v)$ from the properties of NAC-colorings.
	Thus, $h(u)h(v)=(\pi_\red(u), \pi_\blue(u))(\pi_\red(u), \pi_\blue(v))$ which is an edge of $Q$.
	The morphism $h$ is injective by \Cref{rem:cartesianNACcharacterization} since $\delta$ is cartesian.
	Each vertex in $V_{Q_1}\cup V_{Q_2}$ occurs as a coordinate of a vertex in~$h(G)$,
	since $\pi_\red$ and $\pi_\blue$ are surjective.
\end{proof}

\section{Ribbons and parallelogram placements}\label{sec:ribbons}
In this section we describe bracings of graphs (\Cref{sec:bracing}).
We mainly consider a class of graphs (\Cref{sec:classes}) which essentially consists of four-cycles
which we want to place in the plane, forming parallelograms.
Having these 4-cycles in mind we start by defining an equivalence relation on the edges.
The equivalence classes, called \emph{ribbons}, generalize the notion of rows and columns in a rectangular grid.
Ribbons are a concept that is also used in other places under various names (stripes, worms, de Bruijn lines)
and for different purpose (see for instance \cite{Bodini2011,Frettloh2013,Destainville2005,Zhou2019}).

\begin{definition}
	Let $G$ be a graph. Consider the relation on the set of edges, where
	two edges are in relation if they are opposite edges of a 4-cycle subgraph of $G$.
	An equivalence class of the reflexive-transitive closure of the relation is called a \emph{ribbon}.
	\Cref{fig:ribbons} shows all ribbons for some small graphs.
	A ribbon $r$ is \emph{simple} if the subgraph induced by $r$ does not contain any 4-cycle
	(see \Cref{fig:nonsimpleribbon} for an example of a non-simple ribbon).
\end{definition}
	\begin{figure}[ht]
	  \centering
	  \begin{tikzpicture}
	    \node[gvertex] (a) at (0,0) {};
	    \node[gvertex] (b) at (1,0) {};
	    \node[gvertex] (c) at (1,1) {};
	    \node[gvertex] (d) at ($(a)+(c)-(b)$) {};
	    \node[gvertex,rotate around=30:(b)] (e) at ($(b)+(0.8,0)$) {};
	    \node[gvertex] (f) at ($(e)+(c)-(b)$) {};
	    \node[gvertex,rotate around=-20:(e)] (g) at ($(e)+(1.1,0)$) {};
	    \node[gvertex] (h) at ($(g)+(f)-(e)$) {};
	    \node[gvertex,rotate around=-10:(c)] (i) at ($(c)+(0,0.9)$) {};
	    \node[gvertex] (j) at ($(f)+(i)-(c)$) {};
	    
	    \draw[edge] (a)--(b) (b)--(c) (c)--(d) (d)--(a) (b)--(e) (e)--(f) (f)--(c) (e)--(g) (g)--(h) (h)--(f) (c)--(i) (f)--(j) (i)--(j);
	    \draw[ribbon,col1] ($(a)!0.5!(d)$) -- ($(b)!0.5!(c)$) -- ($(e)!0.5!(f)$) -- ($(g)!0.5!(h)$);
	    \draw[ribbon,col2] ($(a)!0.5!(b)$) -- ($(d)!0.5!(c)$);
	    \draw[ribbon,col3] ($(b)!0.5!(e)$) -- ($(c)!0.5!(f)$) -- ($(i)!0.5!(j)$);
	    \draw[ribbon,col6] ($(e)!0.5!(g)$) -- ($(f)!0.5!(h)$);
	    \draw[ribbon,col4] ($(c)!0.5!(i)$) -- ($(f)!0.5!(j)$);
	  \end{tikzpicture}
	  \qquad
	  \begin{tikzpicture}
	    \node[gvertex] (a) at (0,0) {};
	    \node[gvertex] (b) at (1,0) {};
	    \node[gvertex,rotate around=40:(b)] (c) at ($(b)+(0,1)$) {};
	    \node[gvertex] (d) at ($(a)+(c)-(b)$) {};
	    \node[gvertex,rotate around=-10:(b)] (e) at ($(b)+(0,1)$) {};
	    \node[gvertex] (f) at ($(c)+(e)-(b)$) {};
	    \node[gvertex,rotate around=-45:(b)] (g) at ($(b)+(0,1)$) {};
	    \node[gvertex] (h) at ($(e)+(g)-(b)$) {};
	    \node[gvertex,rotate around=-110:(b)] (i) at ($(b)+(0,1)$) {};
	    \node[gvertex] (j) at ($(g)+(i)-(b)$) {};
	    \node[gvertex,rotate around=-180:(b)] (k) at ($(b)+(0,1)$) {};
	    \node[gvertex] (l) at ($(i)+(k)-(b)$) {};
	    \node[gvertex] (m) at ($(a)+(k)-(b)$) {};
	    \draw[edge] (a)--(b) (b)--(c) (c)--(d) (d)--(a) (b)--(e) (e)--(f) (f)--(c) (b)--(g) (g)--(h) (h)--(e) (b)--(i) (i)--(j) (j)--(g) (b)--(k) (k)--(l) (l)--(i) (a)--(m) (m)--(k);
	    \draw[ribbon,col1] ($(a)!0.5!(d)$) -- ($(b)!0.5!(c)$) -- ($(e)!0.5!(f)$);
	    \draw[ribbon,col2] ($(c)!0.5!(f)$) -- ($(b)!0.5!(e)$) -- ($(g)!0.5!(h)$);
	    \draw[ribbon,col3] ($(e)!0.5!(h)$) -- ($(b)!0.5!(g)$) -- ($(i)!0.5!(j)$);
	    \draw[ribbon,col4] ($(g)!0.5!(j)$) -- ($(b)!0.5!(i)$) -- ($(k)!0.5!(l)$);
	    \draw[ribbon,col5] ($(i)!0.5!(l)$) -- ($(b)!0.5!(k)$) -- ($(m)!0.5!(a)$);
	    \draw[ribbon,col6] ($(k)!0.5!(m)$) -- ($(b)!0.5!(a)$) -- ($(c)!0.5!(d)$);
	  \end{tikzpicture}
	  \caption{The ribbons of the graphs are indicated by dashed lines. All edges intersecting the line belong to the same ribbon.}
	  \label{fig:ribbons}
	\end{figure}
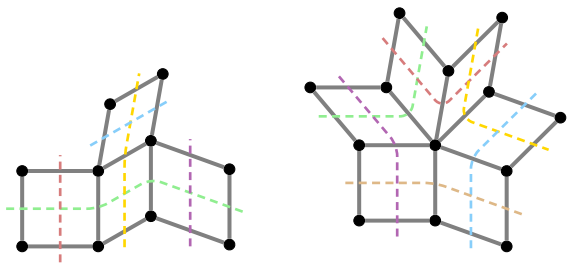
	
	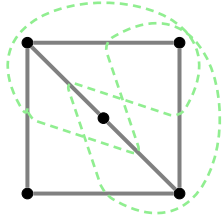
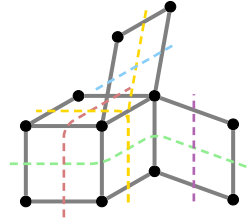
\begin{figure}[ht]
	  \centering
	  \begin{subfigure}[b]{0.45\textwidth}
	    \centering
	    \begin{tikzpicture}
				\clip (-0.3,-0.3) rectangle (2.55,2.55);
				\node[gvertex] (a) at (0,0) {};
				\node[gvertex] (b) at (2,0) {};
				\node[gvertex] (c) at (2,2) {};
				\node[gvertex] (d) at (0,2) {};
				\node[gvertex] (e) at (1,1) {};
				
				\draw[edge] (a)--(b) (b)--(c) (c)--(d) (d)--(a) (b)--(e) (e)--(d);
				\draw[ribbon,col1] ($(a)!0.5!(d)$) -- ($(b)!0.5!(e)$) -- ($(d)!0.5!(c)$) .. controls (3,3.2) and (3,-1.2) .. ($(a)!0.5!(b)$) -- ($(d)!0.5!(e)$) -- ($(b)!0.5!(c)$) .. controls (3.2,3) and (-1.2,3) .. cycle;
			\end{tikzpicture}
			\caption{The only ribbon is not simple.}
			\label{fig:nonsimpleribbon}
	  \end{subfigure}
	  \begin{subfigure}[b]{0.45\textwidth}
			\centering
			\begin{tikzpicture}
				\node[gvertex] (a) at (0,0) {};
				\node[gvertex] (b) at (1,0) {};
				\node[gvertex] (c) at (1,1) {};
				\node[gvertex] (d) at ($(a)+(c)-(b)$) {};
				\node[gvertex,rotate around=30:(b)] (e) at ($(b)+(0.8,0)$) {};
				\node[gvertex] (f) at ($(e)+(c)-(b)$) {};
				\node[gvertex,rotate around=-20:(e)] (g) at ($(e)+(1.1,0)$) {};
				\node[gvertex] (h) at ($(g)+(f)-(e)$) {};
				\node[gvertex,rotate around=-10:(c)] (i) at ($(c)+(0,1.2)$) {};
				\node[gvertex] (j) at ($(f)+(i)-(c)$) {};
				\node[gvertex] (k) at ($(d)+(f)-(c)$) {};
				
				\draw[edge] (a)--(b) (b)--(c) (c)--(d) (d)--(a) (b)--(e) (e)--(f) (f)--(c) (e)--(g) (g)--(h) (h)--(f) (c)--(i) (f)--(j) (i)--(j) (d)--(k) (f)--(k);
				\draw[ribbon,col1] ($(a)!0.5!(d)$) -- ($(b)!0.5!(c)$) -- ($(e)!0.5!(f)$) -- ($(g)!0.5!(h)$);
				\draw[ribbon,col2] ($(a)!0.5!(b)$) -- ($(d)!0.5!(c)$) -- ($(f)!0.5!(k)$);
				\draw[ribbon,col3] ($(b)!0.5!(e)$) -- ($(c)!0.5!(f)$) -- ($(i)!0.5!(j)$);
				\draw[ribbon,col3] ($(b)!0.5!(e)$) -- ($(c)!0.5!(f)$) -- ($(d)!0.5!(k)$);
				\draw[ribbon,col6] ($(e)!0.5!(g)$) -- ($(f)!0.5!(h)$);
				\draw[ribbon,col4] ($(c)!0.5!(i)$) -- ($(f)!0.5!(j)$);
			\end{tikzpicture}
			\caption{The yellow ribbon has three ``ends''.}
			\label{fig:splitribbon}
	  \end{subfigure}
	  \caption{Special cases that might happen for ribbons.}
	\end{figure}
In the case of rectangular grids, there is a natural way how to order the edges in a ribbon, i.e., a row or column.
In our context, there is no natural order of the edges in a ribbon as \Cref{fig:splitribbon} indicates.

From now on, given a walk $W$, the notation $\oriented{u}{v}\in W$ means that the edge $uv$ belongs to $W$ and $u$ precedes $v$ in $W$.
Similarly, for a ribbon $r$, the notation $\oriented{u}{v}\in r \cap W$ means $\oriented{u}{v}\in W$ and $uv\in r$.
If $\oriented{u}{v}\in r \cap W$ is used to iterate in a sum,
the edges of $W$ must be considered as a multiset:
the summand corresponding to $\oriented{u}{v}$ is included as many times as $uv$ occurs in $W$ with $u$ preceding $v$ in $W$.
Similarly for the cardinality of $r\cap W$:
\[
	|r\cap W| = \sum_{\oriented{u}{v}\in r \cap W} 1\,.
\]

Recall that a set of edges $r$ is an \emph{edge cut} of a connected graph $G$
if the graph $G \setminus r = (V_G, E_G \setminus r)$ is disconnected.   
\begin{lemma}
	\label{lem:twoComponents}
	Let $G$ be a connected graph with a simple ribbon $r$, which is an edge cut.
	Then $G\setminus r$ has exactly two connected components.
	In particular, if $w,w'\in V_G$ and $W$ is a walk from $w$ to $w'$, then $|r\cap W|$
	is odd if and only if $r$ separates $w$ and $w'$, i.e.,
	$w$ and $w'$ are in the different connected components of $G\setminus r$.
\end{lemma}
\begin{proof}
	Let $uv$ be an edge of $r$.
	For every edge $u'v'\in r$,
	there exists a sequence of edges $u_1v_1, \dots, u_kv_k$ such that
	$u=u_1, v=v_1, u_k=u', v_k=v'$ and $(u_i,v_i,v_{i+1}, u_{i+1})$ is a 4-cycle in $G$.
	Hence, there are walks $(u_1, \dots, u_k)$ and $(v_1, \dots, v_k)$ in $G$.
	An edge $u_iu_{i+1}$ is in $r$ if and only if $v_iv_{i+1}$ is in $r$.
	But if $u_iu_{i+1},v_iv_{i+1}\in r$, then $(u_i,v_i,v_{i+1}, u_{i+1})$
	would be a 4-cycle in the subgraph induced by $r$, which is not possible since $r$ is simple.
	Hence, no edge of the two walks is in $r$.
	This shows that every vertex of an edge in $r$ is either connected to $u$, or $v$ in $G\setminus r$,
	thus, $G\setminus r$ has two connected components.
	
	If $W$ is a walk from $w$ to $w'$, then $|W\cap r|$ is even if and only 
	if $w$ and $w'$ are in the same connected component of $G\setminus r$.
\end{proof}

We want to consider graphs that somehow consist of parallelograms.
For interpreting this idea we need to look at frameworks rather than graphs.
\begin{definition}
	Let $G$ be a connected graph.
	A placement $\rho:V_G\rightarrow \RR^2$ for $G$ such that $\rho$ 
	is injective and each 4-cycle in $G$ forms a parallelogram in $\rho$
	is called a \emph{parallelogram placement}.
\end{definition}

\begin{remark}
	\label{rem:ribbonsAreParallel}
	Let $\rho$ be a parallelogram placement of a connected graph $G$.
	Edges of a ribbon of $G$ are parallel line segments of the same length in $\rho$.
\end{remark}

\begin{remark}
	\label{rem:parallelogramPlacementImpliesSimpleRibbons} 
	By \Cref{rem:ribbonsAreParallel}, if there was a 4-cycle induced by a ribbon,
	then it would be a degenerate rhombus in a parallelogram placement,
	which contradicts injectivity of the placement.	
	Hence, if a graph allows a parallelogram placement, then all its ribbons are simple.
\end{remark}

The following properties of parallelogram placements are needed later on.
\begin{lemma}
	\label{lem:ribbonTranslation}
	Let $G$ be a connected graph with a parallelogram placement $\rho$
	and ribbon~$r$ which is an edge cut.
	If the vertex set of $r$ is $V_1\cup V_2$,
	where all vertices of $V_i$ belong to the same connected component of $G\setminus r$,
	then $\rho(V_2)$ is a translation of $\rho(V_1)$. 
	In particular, the vector $\rho(u_2)-\rho(u_1)$ is the same for all edges $u_1u_2\in r$, $u_i\in V_i$.
\end{lemma}
\begin{proof}
	The ribbon $r$ is simple by \Cref{rem:parallelogramPlacementImpliesSimpleRibbons}.
	\Cref{lem:twoComponents} gives the partition $V_1\cup V_2$ with $|V_1|=|V_2|$.
	The vector $\rho(u_2)-\rho(u_1)$ is the same for all $u_1u_2\in r$, $u_i\in V_i$,
	by \Cref{rem:ribbonsAreParallel}. 
\end{proof}

\begin{lemma}
	\label{lem:sumOfVectorsInRibbon}
	Let $G$ be a connected graph with a parallelogram placement $\rho$.
	Let $r$ be a ribbon of $G$ which is an edge cut and $W$ be a walk in $G$.
	If $|r\cap W|$ is even, then
	\begin{equation*}
		\sum_{\substack{\oriented{w_1}{w_2}\in r \cap W }} (\rho(w_2)-\rho(w_1)) = 0\,.
	\end{equation*}
\end{lemma}
\begin{proof}
	Let $W=(u_0, u_1, \dots, u_m)$ be a walk.
	All ribbons are simple by \Cref{rem:parallelogramPlacementImpliesSimpleRibbons}.
	Let $V_1\cup V_2$ be as in \Cref{lem:ribbonTranslation}.
	Let the edges of $W$ that are in $r$ be $u_{j_1}u_{j_1+1},\dots, u_{j_k}u_{j_k+1}$
	with $j_1<j_2<\dots<j_k$, $k$ is even by assumption.
	We have that $u_{j_1}, u_{j_2+1},u_{j_3}\dots, u_{j_k+1}\in V_1$
	and $u_{j_1+1}, u_{j_2},u_{j_3+1}\dots, u_{j_k}\in V_2$.
	By \Cref{lem:ribbonTranslation},
	\begin{equation*}
		\sum_{i=1}^k \left(\rho(u_{j_i+1})-\rho(u_{j_i})\right) = 0\,.
	\end{equation*}
\end{proof}

\subsection{Frameworks and graphs consisting of parallelograms}\label{sec:classes}
We now consider classes of graphs which have parallelogram placements.
For this we use three different approaches, each having advantages.
The main property for a graph is the existence of a parallelogram placement.
This existence yields a so called \Pframework.
An illustrating approach is to start from a set of connected parallelograms with additional properties and form a graph.
This will be a \parallelogramTilingFramework.
Finally, we also present a recursive construction for a class of graphs which have a parallelogram placement.
Furthermore, we describe the relations between the different approaches.

\begin{definition}
	A graph $G$ is called \emph{ribbon-cutting graph} if it is connected and every ribbon is an edge cut.
	If $\rho$ is a parallelogram placement of $G$, we call the framework $(G,\rho)$ a \emph{\Pframework{}}.
\end{definition}
A rectangular lattice graph (grid graph) with its natural placement is a \Pframework.
as well as the frameworks in \Cref{fig:introcarpet} and the graphs in \Cref{fig:ribbons,fig:splitribbon} with the placements given by their layouts.

There are ribbon-cutting graphs without any parallelogram placement.
\Cref{fig:nopplacement} shows such a graph, for which the non-existence of a parallelogram placement
follows from failing one of the necessary conditions given by \Cref{thm:ribbonCutGraphProperties}.
On the other hand, the graph in \Cref{fig:nonSeparatingRibbon} is not ribbon-cutting
but has a parallelogram placement.

\begin{figure}[ht]
  \centering
  \begin{tikzpicture}
    \node[fvertex] (a) at (-0.5,0) {};
    \node[fvertex] (b) at (1.3,0) {};
    \node[fvertex] (c) at (1.3,1) {};
    \node[fvertex] (d) at ($(a)+(c)-(b)$) {};
    \node[fvertex,rotate around=30:(b)] (e) at ($(b)+(0.8,0)$) {};
    \node[fvertex] (f) at ($(e)+(c)-(b)$) {};
    \node[fvertex] (g) at ($(d)+(f)-(c)$) {};
    \node[fvertex] (h) at ($(a)+(g)-(d)$) {};
    \node[fvertex,colR] (ip) at ($(b)+(h)-(a)-(0.25,0)$) {};
    
    \draw[edge] (a)--(b) (b)--(c) (c)--(d) (d)--(a) (b)--(e) (e)--(f) (f)--(c) (d)--(g) (f)--(g) (a)--(h) (g)--(h);
    \draw[edge,colR] (b)--(ip) (ip)--(h);
    \draw[ribbon,col1] ($(g)!0.5!(h)$) -- ($(a)!0.5!(d)$) -- ($(b)!0.5!(c)$) -- ($(e)!0.5!(f)$);
    \draw[edge] (c)--(d);
    \draw[ribbon,col3] ($(ip)!0.5!(h)$) -- ($(a)!0.5!(b)$) -- ($(c)!0.5!(d)$) -- ($(f)!0.5!(g)$);
    \draw[ribbon,col4] ($(b)!0.5!(e)$) -- ($(c)!0.5!(f)$) -- ($(d)!0.5!(g)$) -- ($(a)!0.5!(h)$) -- ($(b)!0.5!(ip)$);
  \end{tikzpicture}
  \caption{The graph of the framework is ribbon-cutting, 
  but it has no parallelogram placement:
  if the red vertex and edges were placed forming a parallelogram, two vertices would coincide. 
  \Cref{thm:ribbonCutGraphProperties} also shows this,
  since the two vertices are not separated by a ribbon.}
  \label{fig:nopplacement}
\end{figure}
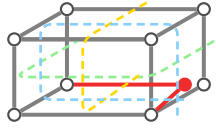

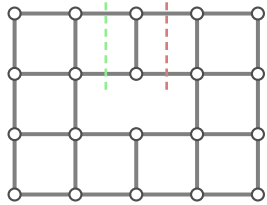
\begin{figure}[ht]
  \centering
  \begin{tikzpicture}[scale=0.8]
	\foreach \j in {0,1,3,4}
	{
		\draw[edge] ($(\j,0)$)--($(\j,3)$0);
	}
	\draw[edge] (2,0)--(2,1) (2,2)--(2,3);
	\foreach \j in {0,1,2,3}
	{
		\draw[edge] ($(0,\j)$)--($(4,\j)$0);
		\foreach \i in {0,1,2,3,4}
		{
			\node[fvertex] (a) at ($\i*(1,0)+\j*(0,1)$) {};
		};
	}
	\draw[ribbon,col1] (1.5, 2) -- (1.5,3);
	\draw[ribbon,col2] (2.5, 2) -- (2.5,3);
  \end{tikzpicture}
  \caption{A parallelogram placement of a graph with ribbons that are
  not edge cuts.}
  \label{fig:nonSeparatingRibbon}
\end{figure}

\begin{theorem}
	\label{thm:ribbonCutGraphProperties}
	If $(G,\rho)$ is a \Pframework,
	then there are no 3-cycles in $G$ and every two vertices are separated by a ribbon.
\end{theorem}
\begin{proof}
	Suppose for contradiction that there is a 3-cycle $(u,v,w)$ in $G$.
	All edges of the 3-cycle must be in the same ribbon, otherwise there is a ribbon which is not an edge cut.
	By~\Cref{rem:ribbonsAreParallel}, the line segments $\rho(u)\rho(v), \rho(u)\rho(w)$ and $\rho(v)\rho(w)$
	are parallel and have the same lengths, which is not possible in the triangle.
	  
	Let $u$ and $v$ be two distinct vertices.
	\Cref{rem:parallelogramPlacementImpliesSimpleRibbons} guarantees that all ribbons are simple.
	Let $W=(u=u_0, u_1, \dots, u_m=v)$ be a walk.
	Let $R$ be the set of ribbons which contain at least one edge of $W$.
	Since $u$ and $v$ are distinct and $\rho$ is injective, we have
	\begin{equation*}
		0 \neq \rho(u_m) - \rho(u_0) = \sum_{i=1}^m \left(\rho(u_i)-\rho(u_{i-1})\right) 
			= \sum_{r\in R} \,\, \sum_{\substack{\oriented{w_1}{w_2}\in r \cap W }} (\rho(w_2)-\rho(w_1))\,.
	\end{equation*}
	All ribbons $r$ such that $|r\cap W|$ is even have a zero contribution by \Cref{lem:sumOfVectorsInRibbon}.
	Hence, there must be a ribbon $r'$ such that $|r' \cap W|$ is odd.
	The ribbon $r'$ separates $u$ and $v$ by \Cref{lem:twoComponents}. 
\end{proof}

The following definition is a slight generalization of the class of graphs used in \cite{DuarteFrancis} using parallelograms instead of rhombi.
\begin{definition}
	\label{def:carpetFramework}
	Let $S$ be a finite set of arbitrary parallelograms in $\RR^2$ (including interiors) such that:
	\begin{itemize}
	  \item if a point belongs to two parallelograms, then it is either a vertex of both, or an interior point of an edge of both,
	  \item if a point belongs to more than two parallelograms, then it is vertex of all of them,
	  \item the boundary of the union $\bigcup S$ is a simple polygon. 
	\end{itemize}
	The framework obtained by taking the 1-skeleton of $S$ together with the vertex positions
	is called a \emph{\parallelogramTilingFramework} (see \Cref{fig:introcarpet} for an example).
\end{definition}

We define a more general class of graphs having the ribbon-cutting property
than the underlying graphs of \parallelogramTilingFramework{}s.
The definition is done recursively adding vertices in a way that a parallelogram placement can be extended
(as we will see in \Cref{lem:placementExtension}).
\begin{definition}
	\label{def:Grec}
	We define the class of graphs $\Grec$ recursively.
	The 4-cycle graph is in~$\Grec$.
	There are two types of construction (see also \Cref{fig:construct}):
	\begin{description}[font=\normalfont]
	  \item[\addPar :] If $G\in\Grec$ with $uv\in E_G$,
		then $(V_G\cup\{w_1, w_2\}, E_G\cup\{uw_1, w_1w_2, w_2v\})$,
		where $w_1,w_2\notin V_G$, is in $\Grec$.
	  \item[\closePar :] If $G\in\Grec$ with $uv, vw\in E_G$ and
	  	the vertex $v$ is separated from any vertex in $V_G\setminus\{u,v,w\}$
	  	by a ribbon which does not contain $uv$ or $vw$, then
			the graph $(V_G\cup\{w'\}, E_G\cup\{uw', w'w\})$, where $w'\notin V_G$, 
			is in $\Grec$.\\
			Note that the separation assumption is needed for avoiding situations as described in \Cref{fig:nopplacement}.
	\end{description}
\end{definition}

\begin{figure}[htb]
	\centering
	\begin{tikzpicture}[scale=1.2]
		\draw[genericgraph,pattern=north east lines,pattern color=black!10!white] (0.5,-0.2) circle[x radius=1cm, y radius=0.6cm];
		\draw[genericgraph,pattern=north east lines,pattern color=black!10!white] (3.5,-0.2) circle[x radius=1cm, y radius=0.6cm];
		\node[gvertex,indicatededge=south,rotate=-90,label={[labelsty]below:$u$}] (a) at (0,0) {};
		\node[gvertex,indicatededge=south,rotate=-90,label={[labelsty]above:$v$}] (b) at (1,0) {};
		\node[gvertex,indicatededge=south,rotate=-90,label={[labelsty]below:$u$}] (c) at (3,0) {};
		\node[gvertex,indicatededge=south,rotate=-90,label={[labelsty]above:$v$}] (d) at (4,0) {};
		\node[ngvertex,label={[labelsty]left:$w_1$}] (e) at (3.2,0.8) {};
		\node[ngvertex,label={[labelsty]right:$w_2$}] (f) at (4.2,0.8) {};
		
		\draw[edge] (a)edge(b) (c)edge(d);
		\draw[nedge] (e)edge(c) (f)edge(d) (e)edge(f);
		\draw[ultra thick,->] (1.75,-0.2) -- (2.25,-0.2);
		\node[labelsty] at (2,-0.9) {\addPar};
	\end{tikzpicture}
	\qquad
	\begin{tikzpicture}[scale=1.2]
		\draw[genericgraph,pattern=north east lines,pattern color=black!10!white] (0.5,-0.2) circle[x radius=1cm, y radius=0.6cm];
		\draw[genericgraph,pattern=north east lines,pattern color=black!10!white] (3.5,-0.2) circle[x radius=1cm, y radius=0.6cm];
		\node[gvertex,indicatededge=south,rotate=-90,label={[labelsty]above:$u$}] (a) at (0,0.2) {};
		\node[gvertex,indicatededge=south,rotate=-90,rotate around=-40:(a),label={[labelsty,label distance=-2pt]90:$v$}] (b) at ($(a)+(0.8,0)$) {};
		\node[gvertex,indicatededge=south,rotate=-90,rotate around=35:(b)] (c) at ($(b)+(0.6,0)$) {};
		\node[gvertex,label={[labelsty]180:$w$}] at (c) {};
		\node[gvertex,indicatededge=south,rotate=-90,label={[labelsty]above:$u$}] (d) at ($(a)+(3,0)$) {};
		\node[gvertex,indicatededge=south,rotate=-90,rotate around=-40:(d),label={[labelsty,label distance=-2pt]90:$v$}] (e) at ($(d)+(0.8,0)$) {};
		\node[gvertex,indicatededge=south,rotate=-90,rotate around=35:(e)] (f) at ($(e)+(0.6,0)$) {};
		\node[gvertex,label={[labelsty]180:$w$}] at (f) {};
		\node[ngvertex,label={[labelsty]above:$w'$}] (g) at ($(d)+(f)-(e)$) {};
		
		\draw[edge] (a)edge(b) (b)edge(c) (d)edge(e) (e)edge(f);
		\draw[nedge] (g)edge(d) (g)edge(f);
		\draw[ultra thick,->] (1.75,-0.2) -- (2.25,-0.2);
		\node[labelsty] at (2,-0.9) {\closePar};
	\end{tikzpicture}
	\caption{Two recursive $\Grec$ constructions.}
	\label{fig:construct}
\end{figure}
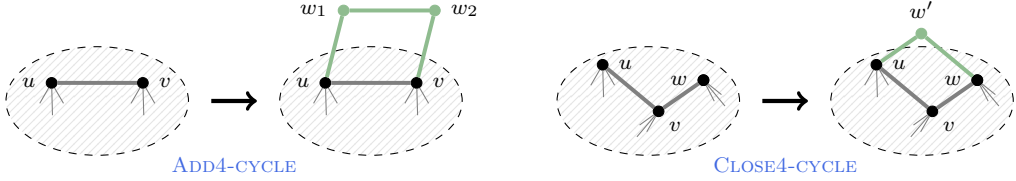
\Cref{fig:splitribbon} gives an example of a graph in $\Grec$
that is not the underlying graph of a \parallelogramTilingFramework. 
It is easy to use the construction to show that the class has the ribbon-cutting property.
\begin{proposition}
	\label{prop:GrecIsRibbonCutting}
	Every graph in $\Grec$ is ribbon-cutting.
\end{proposition}
\begin{proof}
	By structural induction:
	the 4-cycle graph is ribbon-cutting.
	\addPar{} preserves the property
	since $\{uw_1, w_2v\}$ is a new ribbon and $w_1w_2$ belongs to the ribbon of~$uv$.
	\closePar{} does so as well:
	the edges $uw'$ and $w'w$ belong to the ribbons of
	$vw$ and $uv$ respectively.
	If any ribbon of the extended graph were not an edge cut,
	than it would not be an edge cut in the original graph.
	Notice that the separation assumption is not needed for this.
\end{proof}

Recall that for a \Pframework{} $(G,\rho)$, any ribbon $r$ is simple
by \Cref{rem:parallelogramPlacementImpliesSimpleRibbons} and 
$G\setminus r$ has two connected components by \Cref{lem:twoComponents}.
This allows us to translate the vertices of one of the components by a constant vector.
\begin{remark}
	\label{rem:translateComponent}
	Let $(G,\rho)$ be a \Pframework{} and $r$ be a ribbon of $G$.
	Let $V_1$ and $V_2$ be the vertex sets of the two connected components of $G\setminus r$.
	For every vector $t\in\RR^2 \setminus \{\rho(u_1)-\rho(u_2)\colon u_1\in V_1, u_2\in V_2\}$,
	the placement $\rho'$ of~$G$ given by $\rho'(v)=\rho(v) + t$ if $v \in V_2$
	and $\rho'(v)=\rho(v)$ otherwise
	is a parallelogram placement.
\end{remark}

We are going to show the relation between \Pframework{}s, \parallelogramTilingFramework{}s and the graphs in $\Grec$.
Namely, the underlying graphs of \parallelogramTilingFramework{}s are in $\Grec$, 
which is in turn a subset of the underlying graphs of \Pframework{}s. 
For this we need an equivalent condition to the separation assumption in \closePar{}.
\begin{lemma}
	\label{lem:placementExtension}
	For a \Pframework{} $(G,\rho)$ and $uv, vw\in E_G$,
	the following are equivalent:
	\begin{enumerate}
	  \item The vertex $v$ is separated from any vertex in $V_G\setminus\{u,v,w\}$
	  	by a ribbon which does not contain $uv$ or $vw$.
	  \item There exists a parallelogram placement $\rho'$ of the graph 
	  	$G'=(V_G\cup\{w'\}, E_G\cup\{uw', w'w\})$, where $w'\notin V_G$.
	\end{enumerate}
\end{lemma}
\begin{proof}
	$1. \implies 2.$
	If we want to extend $\rho$ to a parallelogram placement of~$G'$,
	the position $\rho(w')$ of the new vertex $w'$ is uniquely determined 
	by the requirement that $(\rho(u),\rho(v),\rho(w),\rho(w'))$ is a parallelogram.
	We can assume that $\rho(u),\rho(v),\rho(w)$ are not collinear,
	hence, $\rho(v)\neq\rho(w')$. If it is not so, we replace $\rho$
	by a parallelogram placement obtained by \Cref{rem:translateComponent} for the ribbon of $uv$
	and a non-zero translation.
	  
	If $\rho:V_{G'}\rightarrow \RR^2$ is injective, we are done. 
	Otherwise, $\rho(w')=\rho(u')$ for a unique vertex $u'\in V_G\setminus\{u,v,w\}$.
	By assumption, there is a ribbon $r$ separating $v$ from $u$ such that $uv,vw\notin r$.
	Thus, $u,v,w$ are in the same connected component of $G\setminus r$,
	whereas $u'$ is in the other one.
	Using \Cref{rem:translateComponent}, there is a parallelogram placement $\rho'$ of $G$
	such that $\rho(w')\neq\rho'(u')$. Moreover,
	the translation vector can be chosen so that the whole image $\rho'(V_G)$ avoids $\rho(w')$.
	Therefore, $\rho'$ uniquely extends to a parallelogram placement of $G'$ by setting $\rho'(w')=\rho(w')$.
	
	$\neg 1. \implies \neg 2.$
	Assume that $u'\in V_G\setminus\{u,v,w\}$ is a vertex such that it is separated from~$v$
	only by the ribbon of $uv$ or $vw$.
	Let $W=(v=u_0, u_1, \dots, u_m=u')$ be a walk from $v$ to $u'$.
	Let $R$ be the set of ribbons which contains at least one edge of $W$.
	All ribbons are simple by \Cref{rem:parallelogramPlacementImpliesSimpleRibbons}.
	By the assumption and \Cref{lem:twoComponents}, $|r\cap W|$ is even for every ribbon $r$ avoiding $uv$ and $vw$.
	For any parallelogram placement $\rho$ of $G$, we have 
	\begin{align*}
		\rho(u') - \rho(v) &\eqwithreference{\phantom{3.7}} \sum_{i=1}^m \left(\rho(u_i)-\rho(u_{i-1})\right) 
			= \sum_{r\in R} \,\, \sum_{\substack{\oriented{w_1}{w_2}\in r \cap W }} (\rho(w_2)-\rho(w_1)) \\
			&\eqwithreference{\ref{lem:sumOfVectorsInRibbon}} \sum_{\substack{r\in R \\ uv \in r \lor vw \in r }} \,\,
				 \sum_{\substack{\oriented{w_1}{w_2}\in r \cap W }} (\rho(w_2)-\rho(w_1)) \\
			&\eqwithreference{\ref{lem:ribbonTranslation}} \alpha(\rho(w)-\rho(v)) + \beta (\rho(u)-\rho(v))\,,
	\end{align*}
	where $\alpha,\beta\in\{0,1\}$. 
	Actually, $\alpha=\beta=1$, otherwise $\rho(u')=\rho(w)$ or $\rho(u')=\rho(u)$, which violates injectivity.
	Hence, $\rho(u') = \rho(w) + \rho(u)-\rho(v)$.
	Assume for contradiction that there is a parallelogram placement $\rho'$ of $G'$.
	Since $\rho'|_{V_G}$ is a parallelogram placement of~$G$,
	we have by the previous $\rho'(u') = \rho'(w) + \rho'(u)-\rho'(v)$.
	But this is a contradiction since $\rho'(w') = \rho'(w) + \rho'(u)-\rho'(v)$ as well and $w'\neq u'$.
\end{proof}

\begin{corollary}
	There exists a \Pframework{} $(G,\rho)$ for every $G\in\Grec$.
\end{corollary}
\begin{proof}
	We proceed by structural induction.
	The 4-cycle can be placed as a parallelogram.
	For a graph $G'$ constructed using \addPar{} from $G$,
	a parallelogram placement of $G$ can be extended to a parallelogram placement of $G'$
	by placing the two new vertices to form a parallelogram so that the placement is injective. 
	If $G'$ is constructed from $G$ by \closePar{},
	then there exists a parallelogram placement of $G'$ by \Cref{lem:placementExtension}.
\end{proof}

\begin{corollary}
	If $(G,\rho)$ is a \parallelogramTilingFramework, then $G\in \Grec$.
	In particular, $(G,\rho)$ is a \Pframework.
\end{corollary}
\begin{proof}
	By the definition of \parallelogramTilingFramework, $\rho$ is a parallelogram placement. 
	Once we show that $G\in \Grec$, the fact that $(G,\rho)$ is a \Pframework{} follows from \Cref{prop:GrecIsRibbonCutting}.
	
	We proceed by induction on the number of parallelograms yielding a \parallelogramTilingFramework{}.
	Let $S$ be the set of parallelograms in $\RR^2$ giving a \parallelogramTilingFramework{} $(G',\rho')$
	according to \Cref{def:carpetFramework}.
	If $|S|=1$, then $(G',\rho')$ is the 4-cycle with a parallelogram placement, hence, $G'\in\Grec$.
	Suppose that $|S|\geq 2$. 
	The boundary of $\bigcup S$ is a simple polygon $M$ with $k$ edges.
	We divide the parallelograms having an edge in the polygon $M$ into the following categories (see \Cref{fig:parallelogramCategories}):
	\begin{itemize}
	  \item $K_1$ --- parallelograms with one edge in $M$,
	  \item $K_2$ --- parallelograms with two incident edges in $M$ such that the vertex that is not in these two edges is not in $M$,
	  \item $K'_2$ --- parallelograms with two incident edges in $M$ that are not in $K_2$,
	  \item $K''_2$ --- parallelograms with two opposite edges in $M$,
	  \item $K_3$ --- parallelograms with three edges in $M$. 
	\end{itemize}
	
	\begin{figure}[ht]
		\centering
		\begin{tikzpicture}
			\node[midfvertex] (a) at (0,0) {};
			\node[midfvertex] (b) at (1.5,0) {};
			\node[midfvertex] (c) at (1.5,1) {};
			\node[midfvertex] (d) at ($(a)+(c)-(b)$) {};
			\node[midfvertex,rotate around=45:(b)] (e) at ($(b)+(0.9,0)$) {};
			\node[midfvertex] (f) at ($(e)+(c)-(b)$) {};
			\node[midfvertex,rotate around=-20:(e)] (g) at ($(e)+(1.2,0)$) {};
			\node[midfvertex] (h) at ($(g)+(f)-(e)$) {};
			\node[midfvertex] (k) at ($(d)+(f)-(c)$) {};
			\node[midfvertex,rotate around=105:(k)] (l) at ($(k)+(1.2,0)$) {};
			\node[midfvertex] (m) at ($(f)+(l)-(k)$) {};
			\node[midfvertex,rotate around=25:(m)] (n) at ($(m)+(0.9,0)$) {};
			\node[midfvertex] (o) at ($(f)+(n)-(m)$) {};
			\node[midfvertex] (p) at ($(b)+(g)-(e)$) {};
			\node[midfvertex,rotate around=-5:(h)] (x) at ($(h)+(1.1,0)$) {};
			\node[midfvertex] (y) at ($(g)+(x)-(h)$) {};
			\node[midfvertex,rotate around=35:(x)] (q) at ($(x)+(1.1,0)$) {};
			\node[midfvertex] (r) at ($(y)+(q)-(x)$) {};
			\node[midfvertex,rotate around=20:(q)] (s) at ($(q)+(0.9,0)$) {};
			\node[midfvertex] (t) at ($(r)+(s)-(q)$) {};
			\node[midfvertex,rotate around=-35:(t)] (u) at ($(t)+(1.3,0)$) {};
			\node[midfvertex] (v) at ($(r)+(u)-(t)$) {};
			\node[midfvertex] (w) at ($(s)+(u)-(t)$) {};
			\node[midfvertex] (z) at ($(y)+(v)-(r)$) {};
			 	
			\draw pic [interiorangle, draw=col1] {angle = c--b--a};
			\draw pic [interiorangle, draw=col1] {angle = b--a--d};
			\draw pic [interiorangle, draw=col1] {angle = a--d--c};
			\draw pic [interiorangle, draw=col1] {angle = p--b--e};
			\draw pic [interiorangle, draw=col1] {angle = g--p--b};
			\draw pic [interiorangle, draw=col1] {angle = e--g--p};
			\draw pic [interiorangle, draw=col1] {angle = w--u--t};
			\draw pic [interiorangle, draw=col1] {angle = s--w--u};
			\draw pic [interiorangle, draw=col1] {angle = t--s--w};
			\draw pic [interiorangle, draw=col1] {angle = z--y--r};
			\draw pic [interiorangle, draw=col1] {angle = v--z--y};
			\draw pic [interiorangle, draw=col1] {angle = r--v--z};
			
			\draw pic [interiorangle, draw=col2] {angle = m--f--k};
			\draw pic [interiorangle, draw=col2] {angle = l--m--f};
			\draw pic [interiorangle, draw=col2] {angle = k--l--m};
			\draw pic [interiorangle, draw=col2] {angle = f--k--l};
			
			\draw pic [interiorangle, draw=col3] {angle = g--h--x};
			\draw pic [interiorangle, draw=col3] {angle = y--g--h};
			\draw pic [interiorangle, draw=col3] {angle = x--y--g};
			\draw pic [interiorangle, draw=col3] {angle = h--x--y};
			
			\draw pic [interiorangle, draw=col4] {angle = c--d--k};
			\draw pic [interiorangle, draw=col4] {angle = d--k--f};
			\draw pic [interiorangle, draw=col4] {angle = e--f--h};
			\draw pic [interiorangle, draw=col4] {angle = f--h--g};
			\draw pic [interiorangle, draw=col4] {angle = r--q--s};
			\draw pic [interiorangle, draw=col4] {angle = q--s--t};
			\draw pic [interiorangle, draw=col4] {angle = u--v--r};
			\draw pic [interiorangle, draw=col4] {angle = t--u--v};
			\draw pic [interiorangle, draw=col4] {angle = y--x--q};
			\draw pic [interiorangle, draw=col4] {angle = x--q--r};
			
			\draw pic [interiorangle, draw=col5] {angle = f--m--n};
			\draw pic [interiorangle, draw=col5] {angle = o--f--m};
			\draw pic [interiorangle, draw=col5] {angle = n--o--f};
			\draw pic [interiorangle, draw=col5] {angle = m--n--o};
			
			\node[category, ecol] at ($(e)!0.5!(c)$) {$K_0$};
			
			\node[category, col1] at ($(a)!0.5!(c)$) {$K_2$};
			\node[category, col1] at ($(e)!0.5!(p)$) {$K_2$};
			\node[category, col1] at ($(w)!0.55!(t)$) {$K_2$};
			\node[category, col1] at ($(y)!0.5!(v)$) {$K_2$};
			
			\node[category, col2] at ($(m)!0.5!(k)$) {$K'_2$};
			
			\node[category, col3] at ($(g)!0.5!(x)$) {$K''_2$};
			
			\node[category, col4] at ($(c)!0.45!(k)$) {$K_1$};
			\node[category, col4] at ($(e)!0.5!(h)$) {$K_1$};
			\node[category, col4] at ($(q)!0.55!(t)$) {$K_1$};
			\node[category, col4] at ($(x)!0.6!(r)$) {$K_1$};
			\node[category, col4] at ($(t)!0.5!(v)$) {$K_1$};
			
			\node[category, col5] at ($(m)!0.5!(o)$) {$K_3$};
			
			\draw[edge] (a)--(b) (b)--(c) (c)--(d) (d)--(a) (b)--(e) (e)--(f) (f)--(c)
			 	(e)--(g) (g)--(h) (h)--(f) (d)--(k) (f)--(k) (k)--(l) (l)--(m) (m)--(f) 
			 	(m)--(n) (n)--(o) (o)--(f) (b)--(p) (p)--(g) (h)--(x) (q)--(r) (y)--(g)
			 	(q)--(s) (s)--(t) (t)--(r) (r)--(v) (v)--(u) (u)--(t) (u)--(w) (w)--(s)
			 	(x)--(y) (x)--(q) (r)--(y) (v)--(z) (z)--(y);
		\end{tikzpicture}
		\caption{An example of dividing parallelograms into categories according to their intersection with the boundary.
			The angles whose contribution to the sum of the interior angles is considered are indicated.
			Note that the parallelogram labeled $K_0$ belongs to S but is not part of the boundary.}
		\label{fig:parallelogramCategories}
	\end{figure}
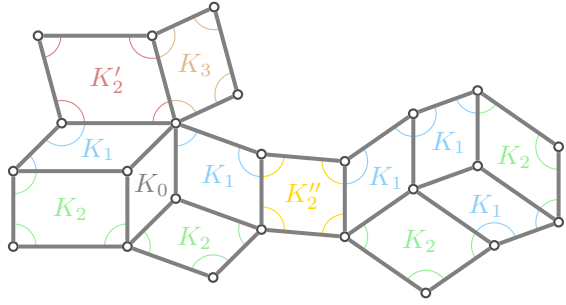
	
	Clearly, $k = |K_1| + 2|K_2| + 2|K'_2| + 2|K''_2| + 3|K_3|$.
	The sum of the interior angles of the simple polygon $M$ equals $(k-2)\pi$.
	Considering contributions to the sum for parallelograms in the categories above
	(see \Cref{fig:parallelogramCategories}), we have
	\begin{align*}
		|K_1|\pi + |K_2|\pi + 2|K'_2|\pi + 2|K''_2|\pi + 2|K_3|\pi &\leq (k-2)\pi  \\
		\iff \qquad 2 &\leq |K_2| + |K_3|\,.
	\end{align*}
	For a parallelogram $s$ in $K_2\cup K_3$,
	$S\setminus s$ satisfies the assumptions of \Cref{def:carpetFramework}.
	Thus, we have a \parallelogramTilingFramework{} $(G,\rho)$ and $G$ is in $\Grec$ by induction assumption.
	If $s\in K_3$, then $G$ can be extended to $G'$ by \addPar.
	If $s\in K_2$, then $G$ can be extended to $G'$ by \closePar,
	since the separation assumption is satisfied by \Cref{lem:placementExtension} and the placement~$\rho'$.  
\end{proof}

\subsection{Bracings}\label{sec:bracing}
A general \Pframework\ is flexible with many degrees of freedom.
By adding edges to the graph we can reduce this number.
In particular we are interested in adding diagonal edges of 4-cycles.
This process is called the bracing of the graph or framework.
\begin{definition}
	A \emph{braced ribbon-cutting graph} is a graph $G=(V_G,E_c\cup E_d)$ where $E_c$ and $E_d$ are two non-empty disjoint sets 
	such that the graph $(V_G,E_c)$ is a ribbon-cutting graph
	and the edges in $E_d$ correspond to diagonals of some 4-cycles of $(V_G,E_c)$.
	These diagonals are also called \emph{braces}.
	If $r$ is a ribbon of $(V_G,E_c)$, then 
	\begin{equation*}
		r \cup \{u_1u_3 \in E_d \colon \exists \text{ 4-cycle } (u_1,u_2,u_3,u_4) \text{ of } (V_G,E_c) \text{ s.t. }  u_1u_2,u_3u_4\in r\}
	\end{equation*}
	 is a \emph{ribbon} of the braced ribbon-cutting graph~$G$.
	
	The framework $(G,\rho)$ is called \emph{braced \Pframework{}} 
	if $G$ is a braced ribbon-cutting graph and $\rho$ is a parallelogram placement for $(V_G,E_c)$.
	\Cref{fig:bracedPframework} shows an example.
\end{definition}

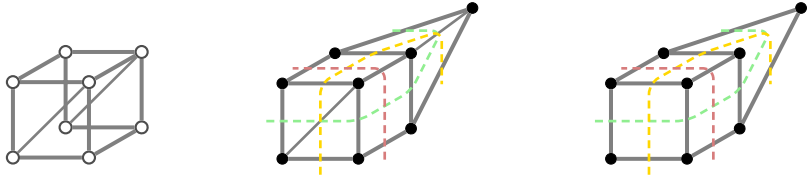
\begin{figure}[ht]
  \centering
  \begin{tikzpicture}
	    \node[fvertex] (a) at (0,0) {};
	    \node[fvertex] (b) at (1,0) {};
	    \node[fvertex] (c) at (1,1) {};
	    \node[fvertex] (d) at ($(a)+(c)-(b)$) {};
	    \node[fvertex,rotate around=30:(b)] (e) at ($(b)+(0.8,0)$) {};
	    \node[fvertex] (f) at ($(e)+(c)-(b)$) {};
	    \node[fvertex] (h) at ($(a)+(e)-(b)$) {};
	    \node[fvertex] (k) at ($(d)+(f)-(c)$) {};
	    
	    \draw[edge] (a)--(b) (b)--(c) (c)--(d) (d)--(a) (b)--(e) (e)--(f) (f)--(c) (h)--(k) (h)--(e) (d)--(k) (f)--(k);
	    \draw[brace] (a)--(c) (f)--(h);
	  \end{tikzpicture}
  \qquad\qquad
  \begin{tikzpicture}
	    \node[gvertex] (a) at (0,0) {};
	    \node[gvertex] (b) at (1,0) {};
	    \node[gvertex] (c) at (1,1) {};
	    \node[gvertex] (d) at ($(a)+(c)-(b)$) {};
	    \node[gvertex,rotate around=30:(b)] (e) at ($(b)+(0.8,0)$) {};
	    \node[gvertex] (f) at ($(e)+(c)-(b)$) {};
	    \node[gvertex] (h) at (2.5,2) {};%($(a)+(e)-(b)$) {};
	    \node[gvertex] (k) at ($(d)+(f)-(c)$) {};
	    
	    \draw[edge] (a)--(b) (b)--(c) (c)--(d) (d)--(a) (b)--(e) (e)--(f) (f)--(c) (h)--(k) (h)--(e) (d)--(k) (f)--(k);
	    \draw[brace] (a)--(c) (f)--(h);
	    \draw[ribbon,col1] ($(a)!0.5!(d)$) -- ($(b)!0.5!(c)$) -- ($(e)!0.5!(f)$) -- ($(h)!0.5!(f)$) -- ($(k)!0.5!(h)$);
	    \draw[ribbon,col3] ($(a)!0.5!(b)$) -- ($(d)!0.5!(c)$) -- ($(f)!0.5!(k)$) -- ($(h)!0.5!(f)$) -- ($(h)!0.5!(e)$);
	    \draw[ribbon,col2] ($(b)!0.5!(e)$) -- ($(c)!0.5!(f)$) -- ($(d)!0.5!(k)$);
  \end{tikzpicture}
  \qquad\qquad
  \begin{tikzpicture}
	    \node[gvertex] (a) at (0,0) {};
	    \node[gvertex] (b) at (1,0) {};
	    \node[gvertex] (c) at (1,1) {};
	    \node[gvertex] (d) at ($(a)+(c)-(b)$) {};
	    \node[gvertex,rotate around=30:(b)] (e) at ($(b)+(0.8,0)$) {};
	    \node[gvertex] (f) at ($(e)+(c)-(b)$) {};
	    \node[gvertex] (h) at (2.5,2) {};%($(a)+(e)-(b)$) {};
	    \node[gvertex] (k) at ($(d)+(f)-(c)$) {};
	    
	    \draw[edge] (a)--(b) (b)--(c) (c)--(d) (d)--(a) (b)--(e) (e)--(f) (f)--(c) (h)--(k) (h)--(e) (d)--(k) (f)--(k);
	    \draw[ribbon,col1] ($(a)!0.5!(d)$) -- ($(b)!0.5!(c)$) -- ($(e)!0.5!(f)$) -- ($(h)!0.5!(f)$) -- ($(k)!0.5!(h)$);
	    \draw[ribbon,col3] ($(a)!0.5!(b)$) -- ($(d)!0.5!(c)$) -- ($(f)!0.5!(k)$) -- ($(h)!0.5!(f)$) -- ($(h)!0.5!(e)$);
	    \draw[ribbon,col2] ($(b)!0.5!(e)$) -- ($(c)!0.5!(f)$) -- ($(d)!0.5!(k)$);
  \end{tikzpicture}
  \caption{An example of a braced \Pframework{} (left) with the underlying ribbon-cutting graph (right) 
  and its bracing (middle).}
  \label{fig:bracedPframework}
\end{figure}

\begin{remark}
	\label{rem:separationBraced}
	A ribbon of a braced ribbon-cutting graph $(V,E_c\cup E_d)$ is an edge cut if
	and only if the corresponding ribbon of $(V,E_c)$ is an edge cut.
\end{remark}
We construct a new graph, which encodes the relations between the ribbons, i.e., we ask whether they share 4-cycles.
A subgraph of this graph indicates whether some of the shared 4-cycles is braced.
\begin{definition}
	Let $G$ be a braced ribbon-cutting graph. The \emph{ribbon graph} $\Gamma$ of $G$
	is the graph with the set of vertices being the set of ribbons
	of $G$ and two ribbons $r_1,r_2$ are adjacent if and only if 
	there is a 4-cycle $(u_1, u_2, u_3, u_4)$ in $G$ such that 
	$u_1 u_2, u_3 u_4 \in r_1$ and $u_1 u_4, u_2 u_3 \in r_2$.
	The subgraph $(V_\Gamma, E_b)$ of $\Gamma$, where 
	\begin{equation*}
		E_b = \{r_1r_2\in E_\Gamma \colon r_1 \cap r_2 \text{ is a subset of braces of } G\}\,,
	\end{equation*}
	 is called the \emph{bracing (sub)graph}. See \Cref{fig:ribbonbracinggraph} for an example of these definitions.
	 
	\begin{figure}[ht]
	  \centering
	  \begin{tikzpicture}
			\begin{scope}[yshift=1.9cm]
			  \node[align=center,anchor=south,font=\scriptsize] at (1.35,0) {Ribbon-Cutting Graph};
			  \node[align=center,anchor=south,font=\scriptsize] at (5.35,0) {Braced\\Ribbon-Cutting Graph};
			  \node[align=center,anchor=south,font=\scriptsize] at (8.75,0) {Ribbon Graph};
			  \node[align=center,anchor=south,font=\scriptsize] at (11.75,0) {Bracing Graph};
			\end{scope}

			\begin{scope}
				\begin{scope}
					\node[gvertex] (a) at (0,0) {};
					\node[gvertex] (b) at (1,0) {};
					\node[gvertex] (c) at (1,1) {};
					\node[gvertex] (d) at ($(a)+(c)-(b)$) {};
					\node[gvertex,rotate around=30:(b)] (e) at ($(b)+(0.8,0)$) {};
					\node[gvertex] (f) at ($(e)+(c)-(b)$) {};
					\node[gvertex,rotate around=-20:(e)] (g) at ($(e)+(1.1,0)$) {};
					\node[gvertex] (h) at ($(g)+(f)-(e)$) {};
					\node[gvertex] (j) at ($(b)+(g)-(e)$) {};
					\node[gvertex] (k) at ($(d)+(f)-(c)$) {};
					
					\draw[edge] (a)--(b) (b)--(c) (c)--(d) (d)--(a) (b)--(e) (e)--(f) (f)--(c) (e)--(g) (g)--(h) (h)--(f) (g)--(j) (b)--(j) (d)--(k) (f)--(k);
					\draw[ribbon,col1] ($(a)!0.5!(d)$) -- ($(b)!0.5!(c)$) -- ($(e)!0.5!(f)$) -- ($(g)!0.5!(h)$);
					\draw[ribbon,col2] ($(a)!0.5!(b)$) -- ($(d)!0.5!(c)$) -- ($(f)!0.5!(k)$);
					\draw[ribbon,col3] ($(g)!0.5!(j)$) -- ($(b)!0.5!(e)$) -- ($(c)!0.5!(f)$) -- ($(d)!0.5!(k)$);
					\draw[ribbon,col4] ($(b)!0.5!(j)$) -- ($(e)!0.5!(g)$) -- ($(f)!0.5!(h)$);
				\end{scope}
				
				\begin{scope}[xshift=4cm]
					\node[gvertex] (a) at (0,0) {};
					\node[gvertex] (b) at (1,0) {};
					\node[gvertex] (c) at (1,1) {};
					\node[gvertex] (d) at ($(a)+(c)-(b)$) {};
					\node[gvertex,rotate around=30:(b)] (e) at ($(b)+(0.8,0)$) {};
					\node[gvertex] (f) at ($(e)+(c)-(b)$) {};
					\node[gvertex,rotate around=-20:(e)] (g) at ($(e)+(1.1,0)$) {};
					\node[gvertex] (h) at ($(g)+(f)-(e)$) {};
					\node[gvertex] (j) at ($(b)+(g)-(e)$) {};
					\node[gvertex] (k) at ($(d)+(f)-(c)$) {};
					
					\draw[edge] (a)--(b) (b)--(c) (c)--(d) (d)--(a) (b)--(e) (e)--(f) (f)--(c) (e)--(g) (g)--(h) (h)--(f) (g)--(j) (b)--(j) (d)--(k) (f)--(k);
					\draw[brace] (a)--(c) (b)--(f) (e)--(j);
					\draw[ribbon,col1] ($(a)!0.5!(d)$) -- ($(b)!0.5!(c)$) -- ($(e)!0.5!(f)$) -- ($(g)!0.5!(h)$);
					\draw[ribbon,col2] ($(a)!0.5!(b)$) -- ($(d)!0.5!(c)$) -- ($(f)!0.5!(k)$);
					\draw[ribbon,col3] ($(g)!0.5!(j)$) -- ($(b)!0.5!(e)$) -- ($(c)!0.5!(f)$) -- ($(d)!0.5!(k)$);
					\draw[ribbon,col4] ($(b)!0.5!(j)$) -- ($(e)!0.5!(g)$) -- ($(f)!0.5!(h)$);
				\end{scope}
				
				\begin{scope}[xshift=8.25cm]
					\node[gvertex,col1] (1) at (0,0) {};
					\node[gvertex,col2] (2) at (1,0) {};
					\node[gvertex,col3] (3) at (1,1) {};
					\node[gvertex,col4] (4) at (0,1) {};
					\draw[edge] (1)--(2) (2)--(3) (3)--(4) (4)--(1) (1)--(3);
				\end{scope}
				
				\begin{scope}[xshift=11.25cm]
					\node[gvertex,col1] (1) at (0,0) {};
					\node[gvertex,col2] (2) at (1,0) {};
					\node[gvertex,col3] (3) at (1,1) {};
					\node[gvertex,col4] (4) at (0,1) {};
					\draw[edge] (1)--(2) (3)--(4) (1)--(3);
				\end{scope}
			\end{scope}
	  
			\begin{scope}[yshift=-4cm]
				\begin{scope}
					\node[gvertex] (a) at (0,0) {};
					\node[gvertex] (b) at (1,0) {};
					\node[gvertex] (c) at (1,1) {};
					\node[gvertex] (d) at ($(a)+(c)-(b)$) {};
					\node[gvertex,rotate around=30:(b)] (e) at ($(b)+(0.8,0)$) {};
					\node[gvertex] (f) at ($(e)+(c)-(b)$) {};
					\node[gvertex,rotate around=-20:(e)] (g) at ($(e)+(1.1,0)$) {};
					\node[gvertex] (h) at ($(g)+(f)-(e)$) {};
					\node[gvertex] (k) at ($(d)+(f)-(c)$) {};
					\node[gvertex,rotate around=105:(k)] (l) at ($(k)+(0.9,0)$) {};
					\node[gvertex] (m) at ($(f)+(l)-(k)$) {};
					\node[gvertex,rotate around=25:(k)] (n) at ($(m)+(1.3,0)$) {};
					\node[gvertex] (o) at ($(f)+(n)-(m)$) {};
					
					\draw[edge] (a)--(b) (b)--(c) (c)--(d) (d)--(a) (b)--(e) (e)--(f) (f)--(c) (e)--(g) (g)--(h) (h)--(f) (d)--(k) (f)--(k) (k)--(l) (l)--(m) (m)--(f) (m)--(n) (n)--(o) (o)--(f);
					\draw[ribbon,col1] ($(a)!0.5!(d)$) -- ($(b)!0.5!(c)$) -- ($(e)!0.5!(f)$) -- ($(g)!0.5!(h)$);
					\draw[ribbon,col2] ($(a)!0.5!(b)$) -- ($(d)!0.5!(c)$) -- ($(f)!0.5!(k)$) -- ($(l)!0.5!(m)$);
					\draw[ribbon,col3] ($(b)!0.5!(e)$) -- ($(c)!0.5!(f)$) -- ($(d)!0.5!(k)$);
					\draw[ribbon,col4] ($(e)!0.5!(g)$) -- ($(f)!0.5!(h)$);
					\draw[ribbon,col5] ($(k)!0.5!(l)$) -- ($(f)!0.5!(m)$) -- ($(n)!0.5!(o)$);
					\draw[ribbon,col6] ($(m)!0.5!(n)$) -- ($(f)!0.5!(o)$);
				\end{scope}
				
				\begin{scope}[xshift=4cm]
					\node[gvertex] (a) at (0,0) {};
					\node[gvertex] (b) at (1,0) {};
					\node[gvertex] (c) at (1,1) {};
					\node[gvertex] (d) at ($(a)+(c)-(b)$) {};
					\node[gvertex,rotate around=30:(b)] (e) at ($(b)+(0.8,0)$) {};
					\node[gvertex] (f) at ($(e)+(c)-(b)$) {};
					\node[gvertex,rotate around=-20:(e)] (g) at ($(e)+(1.1,0)$) {};
					\node[gvertex] (h) at ($(g)+(f)-(e)$) {};
					\node[gvertex] (k) at ($(d)+(f)-(c)$) {};
					\node[gvertex,rotate around=105:(k)] (l) at ($(k)+(0.9,0)$) {};
					\node[gvertex] (m) at ($(f)+(l)-(k)$) {};
					\node[gvertex,rotate around=25:(k)] (n) at ($(m)+(1.3,0)$) {};
					\node[gvertex] (o) at ($(f)+(n)-(m)$) {};
					
					\draw[edge] (a)--(b) (b)--(c) (c)--(d) (d)--(a) (b)--(e) (e)--(f) (f)--(c) (e)--(g) (g)--(h) (h)--(f) (d)--(k) (f)--(k) (k)--(l) (l)--(m) (m)--(f) (m)--(n) (n)--(o) (o)--(f);
					\draw[brace] (a)--(c) (e)--(h) (f)--(l);
					\draw[ribbon,col1] ($(a)!0.5!(d)$) -- ($(b)!0.5!(c)$) -- ($(e)!0.5!(f)$) -- ($(g)!0.5!(h)$);
					\draw[ribbon,col2] ($(a)!0.5!(b)$) -- ($(d)!0.5!(c)$) -- ($(f)!0.5!(k)$) -- ($(l)!0.5!(m)$);
					\draw[ribbon,col3] ($(b)!0.5!(e)$) -- ($(c)!0.5!(f)$) -- ($(d)!0.5!(k)$);
					\draw[ribbon,col4] ($(e)!0.5!(g)$) -- ($(f)!0.5!(h)$);
					\draw[ribbon,col5] ($(k)!0.5!(l)$) -- ($(f)!0.5!(m)$) -- ($(n)!0.5!(o)$);
					\draw[ribbon,col6] ($(m)!0.5!(n)$) -- ($(f)!0.5!(o)$);
				\end{scope}
				
				\begin{scope}[xshift=8.75cm,yshift=1cm]
					\coordinate (o) at (0,0);
					\node[gvertex,col1] (1) at (1,0) {};
					\node[gvertex,col2,rotate around=60:(o)] (2) at (1,0) {};
					\node[gvertex,col3,rotate around=120:(o)] (3) at (1,0) {};
					\node[gvertex,col4,rotate around=180:(o)] (4) at (1,0) {};
					\node[gvertex,col5,rotate around=240:(o)] (5) at (1,0) {};
					\node[gvertex,col6,rotate around=300:(o)] (6) at (1,0) {};
					\draw[edge] (1)--(2) (1)--(3) (1)--(4) (2)--(3) (2)--(5) (5)--(6);
				\end{scope}
				
				\begin{scope}[xshift=11.75cm,yshift=1cm]
					\coordinate (o) at (0,0);
					\node[gvertex,col1] (1) at (1,0) {};
					\node[gvertex,col2,rotate around=60:(o)] (2) at (1,0) {};
					\node[gvertex,col3,rotate around=120:(o)] (3) at (1,0) {};
					\node[gvertex,col4,rotate around=180:(o)] (4) at (1,0) {};
					\node[gvertex,col5,rotate around=240:(o)] (5) at (1,0) {};
					\node[gvertex,col6,rotate around=300:(o)] (6) at (1,0) {};
					\draw[edge] (1)--(2) (1)--(4) (2)--(5);
				\end{scope}
			\end{scope}
	  \end{tikzpicture}
	  \caption{Two ribbon-cutting graphs with an example of a bracing as well as the corresponding ribbon graph and bracing graph.
	  The vertices in the ribbon graph and the bracing graph are colored in correspondence with the indicated ribbons.}
	  \label{fig:ribbonbracinggraph}
	\end{figure}
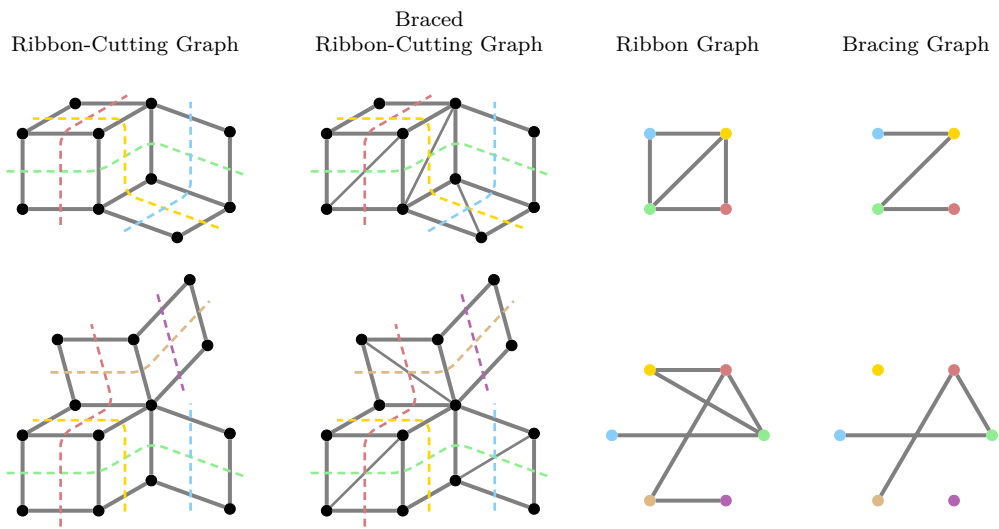
\end{definition}

We remark that the bracing subgraph according to the definition in \cite{DuarteFrancis} does not contain the ribbons which have no brace.
In our definition these ribbons are isolated vertices.

There are no loops in ribbon and bracing graphs if all ribbons of the underlying unbraced ribbon-cutting graph are simple.
An edge in a bracing graph does not determine uniquely a braced 4-cycle (see the yellow and green ribbon in \Cref{fig:bracedPframework}).

Now we have all definitions to recall the main theorem of \cite{DuarteFrancis}.
In the next section we extend this theorem to \Pframework s and also prove the other direction.
\begin{theorem}[\cite{DuarteFrancis}]
	Let $(G, \rho)$ be a braced \parallelogramTilingFramework.
	If the bracing graph of $G$ is a connected,
	then the braced framework is rigid.
\end{theorem}

\section{Flexibility of braced \Pframework{}s}\label{sec:flex}
In this section we determine when a bracing makes the framework rigid and in which cases it remains flexible.
We use cartesian NAC-colorings for that. The theory is therefore based on \cite{flexibleLabelings}.
Indeed, we show that a \Pframework\ is flexible if and only if it has a cartesian NAC-coloring.
This finally leads to a proof of the main theorem.

Cartesian NAC-colorings of a subclass of ribbon-cutting graphs can be characterized using ribbons.
\begin{lemma}
	\label{lem:cartesianIffRibbonsMonochromatic}
	Let $G$ be a braced ribbon-cutting graph such that every two vertices are separated by a ribbon.
	A NAC-coloring of $G$ is cartesian if and only if each ribbon of $G$ is monochromatic. 
\end{lemma}
\begin{proof}
	Let $\delta$ be a NAC-coloring.
	If $\delta$ is cartesian, then all 4-cycles are either monochromatic or opposite edges have the same color.
	Since the edges of a braced 4-cycle have the same color, ribbons are monochromatic.
	On the other hand, if the ribbons are monochromatic, then two vertices cannot be connected
	by a \blue{} and \red{} path simultaneously since they are separated by a ribbon.
\end{proof}

\begin{theorem}
	\label{thm:flexGivesCartNAC}
	If a braced \Pframework{} $(G, \rho)$ is flexible,
	then $G$ has a cartesian NAC-coloring.
\end{theorem}
\begin{proof}
	A NAC-coloring for $G$ can be constructed as in the proof of~\cite[Theorem 3.1]{flexibleLabelings}.
	The existence of a flex implies that there is an irreducible algebraic curve $C$ of placements equivalent to $\rho$.
	For vertices $u,v\in V_G$ and a placement in $C$, let $(x_u,y_u)$ and $(x_v,y_v)$ be the coordinates of $u$ and $v$ in the placement.
	The proof of the theorem defines
	\begin{equation*}
		W_{u,v} = (x_v - x_u) + \ci (y_v - y_u)\,.
	\end{equation*}
	Let $(u_1,u_2,u_3,u_4)$ be a parallelogram. 
  	Since $W_{u_1,u_2}=W_{u_4,u_3}$	for the opposite edges $u_1u_2$ and $u_3u_4$,
	we have that $u_1u_2$ and $u_3u_4$ have the same color in the NAC-coloring constructed as in~\cite[Theorem 3.1]{flexibleLabelings}.
	Therefore, ribbons are monochromatic since a 4-cycle with a diagonal is monochromatic.
	The NAC-coloring is cartesian by \Cref{lem:cartesianIffRibbonsMonochromatic}
	since every two vertices are separated by a ribbon by \Cref{thm:ribbonCutGraphProperties}
	applied to the underlying unbraced \Pframework{} and \Cref{rem:separationBraced}.
\end{proof}

\begin{lemma}
	\label{lem:monochromaticSumInvariant}
	Let $(G, \rho)$ be a \Pframework.
	Let $u,v\in V_G$
	and $W,W'$ be walks from $u$ to~$v$ in $G$.
	If $G$ has a cartesian NAC-coloring $\delta$ and $c\in\{\red, \blue\}$, then
	\begin{equation*}
		\sum_{\substack{\oriented{w_1}{w_2}\in W \\ \delta(w_1w_2)=c }} (\rho(w_2)-\rho(w_1))
		= \sum_{\substack{\oriented{w_1}{w_2}\in W' \\ \delta(w_1w_2)=c}} (\rho(w_2)-\rho(w_1))\,.
	\end{equation*}
\end{lemma}
\begin{proof}
	Let $\widehat{W}$ be the walk obtained by concatenating $W$ and the inverse of $W'$.
	We consider the sum
	\begin{equation*}
		\sum_{\substack{\oriented{w_1}{w_2}\in \widehat{W} \\ \delta(w_1w_2)=c }} (\rho(w_2)-\rho(w_1)) \,.
	\end{equation*}
	Since $\widehat{W}$ is closed and ribbons are simple by \Cref{rem:parallelogramPlacementImpliesSimpleRibbons},
	$|r\cap \widehat{W}|$ is even for every ribbon~$r$ (this is a consequence of \Cref{lem:twoComponents}).
	As each ribbon $r$ is monochromatic in a cartesian NAC-coloring, the number of edges in $r\cap \widehat{W}$
	included in the sum is even. Hence, the sum is zero by \Cref{lem:sumOfVectorsInRibbon}.
\end{proof}
Using the lemma we show the reverse direction of \Cref{thm:flexGivesCartNAC}.
The proof is constructive, i.e., it provides a flex.

\begin{theorem}
	\label{thm:cartNACimpliesFlex}
	If a braced \Pframework{} has a cartesian NAC-coloring, then it is flexible.
\end{theorem}
\begin{proof}
	Let $(G',\rho)$ be a braced \Pframework{} and $\delta'$ be a cartesian NAC-coloring of $G'$.
	We can assume that $\rho(\bar{u})=(0,0)$ for a fixed vertex $\bar{u}\in G'$.
	Using the ``zigzag'' grid construction from~\cite{flexibleLabelings},
	there is a proper flexible framework $(G', \rho')$.
	Hence, it is sufficient to show that
	the ``zigzag'' grid can be chosen so that $\rho'=\rho$.
	Let $G$ be the graph~$G'$ with braces removed and $\delta$ be the NAC-coloring of $G$
	obtained by restricting $\delta'$.
	Since monochromatic 4-cycles preserve their shapes during a flex from a ``zigzag'' grid construction,
	the flex of $G$ constructed from $\delta$ is the same as the flex of $G'$ constructed from $\delta'$
	if the same ``zigzag'' grid is used.
	Hence, we have to find a ``zigzag'' grid such that the flex of $G$ obtained by using $\delta$ starts at $\rho$.
	
	Let $R_1, \dots, R_m$, resp.\ $B_1, \dots, B_n$, be the vertex sets
	of the connected components of $G_\red^\delta$, resp.\ $G_\blue^\delta$.
	We define a map $\rho_\red:\{R_1, \dots, R_m\}\rightarrow\RR^2$ as follows:
	for $R_i$, let $W$ be any walk from $\bar{u}$ to a vertex of $G$ in $R_i$ and
	\begin{equation*}
		\rho_\red (R_i) = \sum_{\substack{\oriented{w_1}{w_2}\in W \\ \delta(w_1w_2)=\blue }} (\rho(w_2)-\rho(w_1))\,.
	\end{equation*}
	\Cref{lem:monochromaticSumInvariant} guarantees that it is well-defined,
	namely, the sum is independent of the choice of $W$ and the vertex in $R_i$.
	We define $\rho_\blue:\{B_1, \dots, B_n\}\rightarrow\RR^2$ analogously by swapping \red{} and \blue.
	
	For $t\in[0,2\pi]$ and $v\in V_G$, where $v\in R_i\cap B_j$, let
	\begin{equation*}
		\rho_t(v) = \begin{pmatrix}
		\cos(t) & \sin(t) \\
		-\sin(t) & \cos(t)
		\end{pmatrix}\cdot \rho_\red(R_i) + \rho_\blue(B_j)\,.
	\end{equation*}
	If $W$ is a walk from $\bar{u}$ to $v$, then 
	\begin{align*}
		\rho_0(v) &= \rho_\red(R_i) + \rho_\blue(B_j) 
			= \sum_{\substack{\oriented{w_1}{w_2}\in W \\ \delta(w_1w_2)=\blue }} (\rho(w_2)-\rho(w_1))
			+ \sum_{\substack{\oriented{w_1}{w_2}\in W \\ \delta(w_1w_2)=\red }} (\rho(w_2)-\rho(w_1)) \\
			&= \sum_{\substack{\oriented{w_1}{w_2}\in W }} (\rho(w_2)-\rho(w_1)) = \rho(v) - \rho(\bar{u}) =  \rho(v)\,.
	\end{align*}
	Therefore, $\rho_t$ is a flex of $(G,\rho)$. See \cite{flexibleLabelings} for a proof that the edge lengths are constant
	and no two adjacent vertices are mapped to the same point.
\end{proof}

Finally, we connect the results of flexibility and NAC-colorings
with the connectivity of the bracing graph,
which forms the last part of the proof of \Cref{thm:main_result}.
\begin{theorem}
	\label{thm:bracingGraphConnectedIffNoCartNAC}
	Let $G$ be a braced ribbon-cutting graph such that every two vertices are separated by a ribbon.
	The bracing graph of $G$ is connected if and only if
	$G$ does not have a cartesian NAC-coloring.  
\end{theorem}
\begin{proof}
	Let $B$ be the bracing graph of $G$.
	In a cartesian NAC-coloring $G$, ribbons are monochromatic by \Cref{lem:cartesianIffRibbonsMonochromatic}.
	Hence, if two ribbons are adjacent in~$B$, then the union of their edges is monochromatic.
	Therefore, if $B$ is connected, all edges of $G$ must have the same color,
	namely, no cartesian NAC-coloring exists.
	
	For the opposite implication, assume $B$ is not connected.
	We color the edges of the ribbons of one connected component by \red{} and the rest by \blue.
	To show that this surjective edge coloring is a NAC-coloring, consider a cycle $C$.
	Let $uv$ be an edge of $C$ and $r$ be the ribbon containing $uv$.
	Since $r$ separates $G$, $r$~contains another edge $u'v'$ of~$C$.
	Since ribbons are monochromatic, either all edges of $C$ have the same color or there are two edges of each color.
	The obtained NAC-coloring is cartesian by \Cref{lem:cartesianIffRibbonsMonochromatic}.  
\end{proof}
\begin{proof}[Proof of \Cref{thm:main_result}]
	Let $(G,\rho)$ be a braced \Pframework.
	Every two vertices are separated by a ribbon
	by \Cref{thm:ribbonCutGraphProperties} and \Cref{rem:separationBraced}.
	Hence, $(G,\rho)$ is rigid if and only if $G$ has no cartesian NAC-coloring 
	(\Cref{thm:flexGivesCartNAC,thm:cartNACimpliesFlex}) if and 
	only if the bracing graph of $G$ is connected (\Cref{thm:bracingGraphConnectedIffNoCartNAC}).
	
	Each edge of the bracing graph corresponds to at least one brace.
	The minimum number of braces making the framework rigid follows from the fact that
	the number of edges of a spanning tree of the bracing graph
	is one less than the number of vertices, i.e., ribbons.
	The result is also illustrated in \Cref{fig:finalexample}.
\end{proof}
\begin{figure}[ht]
  \centering
  \begin{tikzpicture}[rs/.style={black!10!white},rsl/.style={black!40!white},fcyc/.style={opacity=0.15}]
		\begin{scope}[rotate=-90]
			\node[fvertex] (0) at (0,0) {};
			\node[fvertex] (1) at (0,1) {};
			\node[fvertex,rotate around=0:(0)] (2) at ($(0)+(1.00,0)$) {};
			\node[fvertex] (3) at ($(2)+(1)-(0)$) {};
			\node[fvertex,rotate around=45:(2)] (4) at ($(2)+(0.50,0)$) {};
			\node[fvertex] (5) at ($(4)+(3)-(2)$) {};
			\node[fvertex,rotate around=120:(1)] (6) at ($(1)+(0.50,0)$) {};
			\node[fvertex] (7) at ($(6)+(3)-(1)$) {};
			\node[fvertex,rotate around=72:(7)] (8) at ($(7)+(1.50,0)$) {};
			\node[fvertex] (9) at ($(8)+(3)-(7)$) {};
			\node[fvertex] (10) at ($(5)+(9)-(3)$) {};
			\node[fvertex] (11) at ($(8)+(10)-(9)$) {};
			\node[fvertex] (12) at ($(6)+(8)-(7)$) {};
			\node[fvertex] (13) at ($(0)+(6)-(1)$) {};
			\node[fvertex,rotate around=144:(6)] (14) at ($(6)+(1.00,0)$) {};
			\node[fvertex] (15) at ($(14)+(13)-(6)$) {};
			\node[fvertex,rotate around=108:(6)] (16) at ($(6)+(1.00,0)$) {};
			\node[fvertex] (17) at ($(16)+(14)-(6)$) {};
			\node[fvertex] (18) at ($(12)+(16)-(6)$) {};
			\node[fvertex] (19) at ($(17)+(18)-(16)$) {};
			\node[fvertex,rotate around=90:(12)] (20) at ($(12)+(0.75,0)$) {};
			\node[fvertex] (21) at ($(20)+(8)-(12)$) {};
			\node[fvertex] (22) at ($(11)+(21)-(8)$) {};
			\node[fvertex] (23) at ($(18)+(20)-(12)$) {};
			\node[fvertex] (24) at ($(23)+(21)-(20)$) {};
			\node[fvertex] (25) at ($(19)+(23)-(18)$) {};
			\node[fvertex] (26) at ($(24)+(22)-(21)$) {};
			\node[fvertex] (27) at ($(25)+(24)-(23)$) {};
			\node[fvertex] (28) at ($(27)+(26)-(24)$) {};
			\node[fvertex] (29) at ($(25)+(28)-(27)$) {};
			\draw[edge] (0)--(1)
					(0)--(2) (2)--(3) (3)--(1)
					(2)--(4) (4)--(5) (5)--(3)
					(1)--(6) (6)--(7) (7)--(3)
					(7)--(8) (8)--(9) (9)--(3)
					(5)--(10) (10)--(9) 
					(8)--(11) (11)--(10) 
					(6)--(12) (12)--(8) 
					(0)--(13) (13)--(6) 
					(6)--(14) (14)--(15) (15)--(13)
					(6)--(16) (16)--(17) (17)--(14)
					(12)--(18) (18)--(16) 
					(17)--(19) (19)--(18) 
					(12)--(20) (20)--(21) (21)--(8)
					(11)--(22) (22)--(21) 
					(18)--(23) (23)--(20) 
					(23)--(24) (24)--(21) 
					(19)--(25) (25)--(23) 
					(24)--(26) (26)--(22) 
					(25)--(27) (27)--(24) 
					(27)--(28) (28)--(26) 
					(25)--(29) (29)--(28)
					;
			\draw[brace] (27)edge(29) (3)edge(4) (1)edge(13) (7)edge(9) (12)edge(16) (20)edge(18) (23)edge(19);
			\draw[ribbon,rs] ($(4)!0.5!(5)$) node[labelsty,rsl,below right=2.5pt] {$r_3$} -- ($(2)!0.5!(3)$) -- ($(0)!0.5!(1)$) -- ($(6)!0.5!(13)$) -- ($(14)!0.5!(15)$);
			\draw[ribbon,rs] ($(5)!0.5!(10)$) node[labelsty,rsl,below right=2.5pt] {$r_4$} -- ($(3)!0.5!(9)$) -- ($(7)!0.5!(8)$) -- ($(6)!0.5!(12)$) -- ($(16)!0.5!(18)$) -- ($(17)!0.5!(19)$);
			\draw[ribbon,rs] ($(10)!0.5!(11)$) node[labelsty,rsl,below right=2.5pt] {$r_5$} -- ($(8)!0.5!(9)$) -- ($(3)!0.5!(7)$) -- ($(1)!0.5!(6)$) -- ($(0)!0.5!(13)$);
			\draw[ribbon,rs] ($(11)!0.5!(22)$) node[labelsty,rsl,below right=5pt] {$r_6$} -- ($(8)!0.5!(21)$) -- ($(12)!0.5!(20)$) -- ($(18)!0.5!(23)$) -- ($(19)!0.5!(25)$);
			\draw[ribbon,rs] ($(22)!0.5!(26)$) node[labelsty,rsl,below right=5pt] {$r_7$} -- ($(21)!0.5!(24)$) -- ($(20)!0.5!(23)$) -- ($(12)!0.5!(18)$) -- ($(6)!0.5!(16)$) -- ($(14)!0.5!(17)$);
			\draw[ribbon,rs] ($(26)!0.5!(28)$) node[labelsty,rsl,below right=5pt] {$r_8$} -- ($(24)!0.5!(27)$) -- ($(23)!0.5!(25)$) -- ($(18)!0.5!(19)$) -- ($(16)!0.5!(17)$) -- ($(6)!0.5!(14)$) -- ($(13)!0.5!(15)$);
			\draw[ribbon,rs] ($(2)!0.5!(4)$) node[labelsty,rsl,left=5pt] {$r_2$} -- ($(3)!0.5!(5)$) -- ($(9)!0.5!(10)$) -- ($(8)!0.5!(11)$) -- ($(21)!0.5!(22)$) -- ($(24)!0.5!(26)$) -- ($(27)!0.5!(28)$) -- ($(25)!0.5!(29)$);
			\draw[ribbon,rs] ($(0)!0.5!(2)$) node[labelsty,rsl,left=5pt] {$r_1$} -- ($(1)!0.5!(3)$) -- ($(6)!0.5!(7)$) -- ($(8)!0.5!(12)$) -- ($(20)!0.5!(21)$) -- ($(23)!0.5!(24)$) -- ($(25)!0.5!(27)$) -- ($(28)!0.5!(29)$);
		\end{scope}
		\begin{scope}[xshift=10cm,scale=0.8]
			\coordinate (o) at (0,0);
		  \node[gvertex,label={[labelsty]0:$r_1$}] (b1) at (1.5,0) {};
		  \node[gvertex,label={[labelsty]0:$r_2$},rotate around=45:(o)] (b2) at (b1) {};
		  \node[gvertex,label={[labelsty]45:$r_3$},rotate around=45:(o)] (b3) at (b2) {};
		  \node[gvertex,label={[labelsty]90:$r_4$},rotate around=45:(o)] (b4) at (b3) {};
		  \node[gvertex,label={[labelsty]135:$r_5$},rotate around=45:(o)] (b5) at (b4) {};
		  \node[gvertex,label={[labelsty]180:$r_6$},rotate around=45:(o)] (b6) at (b5) {};
		  \node[gvertex,label={[labelsty]225:$r_7$},rotate around=45:(o)] (b7) at (b6) {};
		  \node[gvertex,label={[labelsty]270:$r_8$},rotate around=45:(o)] (b8) at (b7) {};
		  \draw[edge] (b1)edge(b2)
		              (b2)edge(b3)
		              (b3)edge(b5)
		              (b4)edge(b5) (b4)edge(b7)
		              (b6)edge(b7) (b6)edge(b8);
		\end{scope}
		\begin{scope}[yshift=-4cm,rotate=-90]
			\node[fvertex] (0) at (0,0) {};
			\node[fvertex] (1) at (0,1) {};
			\node[fvertex,rotate around=0:(0)] (2) at ($(0)+(1.00,0)$) {};
			\node[fvertex] (3) at ($(2)+(1)-(0)$) {};
			\node[fvertex,rotate around=45:(2)] (4) at ($(2)+(0.50,0)$) {};
			\node[fvertex] (5) at ($(4)+(3)-(2)$) {};
			\node[fvertex,rotate around=120:(1)] (6) at ($(1)+(0.50,0)$) {};
			\node[fvertex] (7) at ($(6)+(3)-(1)$) {};
			\node[fvertex,rotate around=72:(7)] (8) at ($(7)+(1.50,0)$) {};
			\node[fvertex] (9) at ($(8)+(3)-(7)$) {};
			\node[fvertex] (10) at ($(5)+(9)-(3)$) {};
			\node[fvertex] (11) at ($(8)+(10)-(9)$) {};
			\node[fvertex] (12) at ($(6)+(8)-(7)$) {};
			\node[fvertex] (13) at ($(0)+(6)-(1)$) {};
			\node[fvertex,rotate around=144:(6)] (14) at ($(6)+(1.00,0)$) {};
			\node[fvertex] (15) at ($(14)+(13)-(6)$) {};
			\node[fvertex,rotate around=108:(6)] (16) at ($(6)+(1.00,0)$) {};
			\node[fvertex] (17) at ($(16)+(14)-(6)$) {};
			\node[fvertex] (18) at ($(12)+(16)-(6)$) {};
			\node[fvertex] (19) at ($(17)+(18)-(16)$) {};
			\node[fvertex,rotate around=90:(12)] (20) at ($(12)+(0.75,0)$) {};
			\node[fvertex] (21) at ($(20)+(8)-(12)$) {};
			\node[fvertex] (22) at ($(11)+(21)-(8)$) {};
			\node[fvertex] (23) at ($(18)+(20)-(12)$) {};
			\node[fvertex] (24) at ($(23)+(21)-(20)$) {};
			\node[fvertex] (25) at ($(19)+(23)-(18)$) {};
			\node[fvertex] (26) at ($(24)+(22)-(21)$) {};
			\node[fvertex] (27) at ($(25)+(24)-(23)$) {};
			\node[fvertex] (28) at ($(27)+(26)-(24)$) {};
			\node[fvertex] (29) at ($(25)+(28)-(27)$) {};
			\draw[bedge] (0)--(2) (3)--(1) (5)--(3) (6)--(7) (12)--(8) (6)--(14) (15)--(13)
					(6)--(16) (16)--(17) (17)--(14)
					(12)--(18) (18)--(16) 
					(17)--(19) (19)--(18) 
					(12)--(20) (20)--(21) (21)--(8)
					(11)--(22) (22)--(21) 
					(18)--(23) (23)--(20) 
					(23)--(24) (24)--(21) 
					(19)--(25) (25)--(23) 
					(24)--(26) (26)--(22) 
					(25)--(27) (27)--(24) 
					(27)--(28) (28)--(26) 
					(25)--(29) (29)--(28)
					;
			\draw[redge] (0)--(1) (2)--(3) (2)--(4) (4)--(5) (5)--(3) (1)--(6)  (7)--(3)
					(7)--(8) (8)--(9) (9)--(3) (5)--(10) (10)--(9) (8)--(11) (11)--(10) 
					(6)--(12) (0)--(13) (13)--(6) (14)--(15) (18)--(16) (17)--(19)
					(22)--(21) (24)--(26) (27)--(28)  (25)--(29);
			\draw[brace,colR] (3)edge(4) (5)edge(9) (8)edge(10);
			\draw[brace,colB] (20)edge(18) (23)edge(19) (23)edge(27);
			\draw[ribbon,rs] ($(4)!0.5!(5)$) node[labelsty,rsl,below right=2.5pt] {$r_3$} -- ($(2)!0.5!(3)$) -- ($(0)!0.5!(1)$) -- ($(6)!0.5!(13)$) -- ($(14)!0.5!(15)$);
			\draw[ribbon,rs] ($(5)!0.5!(10)$) node[labelsty,rsl,below right=2.5pt] {$r_4$} -- ($(3)!0.5!(9)$) -- ($(7)!0.5!(8)$) -- ($(6)!0.5!(12)$) -- ($(16)!0.5!(18)$) -- ($(17)!0.5!(19)$);
			\draw[ribbon,rs] ($(10)!0.5!(11)$) node[labelsty,rsl,below right=2.5pt] {$r_5$} -- ($(8)!0.5!(9)$) -- ($(3)!0.5!(7)$) -- ($(1)!0.5!(6)$) -- ($(0)!0.5!(13)$);
			\draw[ribbon,rs] ($(11)!0.5!(22)$) node[labelsty,rsl,below right=5pt] {$r_6$} -- ($(8)!0.5!(21)$) -- ($(12)!0.5!(20)$) -- ($(18)!0.5!(23)$) -- ($(19)!0.5!(25)$);
			\draw[ribbon,rs] ($(22)!0.5!(26)$) node[labelsty,rsl,below right=5pt] {$r_7$} -- ($(21)!0.5!(24)$) -- ($(20)!0.5!(23)$) -- ($(12)!0.5!(18)$) -- ($(6)!0.5!(16)$) -- ($(14)!0.5!(17)$);
			\draw[ribbon,rs] ($(26)!0.5!(28)$) node[labelsty,rsl,below right=5pt] {$r_8$} -- ($(24)!0.5!(27)$) -- ($(23)!0.5!(25)$) -- ($(18)!0.5!(19)$) -- ($(16)!0.5!(17)$) -- ($(6)!0.5!(14)$) -- ($(13)!0.5!(15)$);
			\draw[ribbon,rs] ($(2)!0.5!(4)$) node[labelsty,rsl,left=5pt] {$r_2$} -- ($(3)!0.5!(5)$) -- ($(9)!0.5!(10)$) -- ($(8)!0.5!(11)$) -- ($(21)!0.5!(22)$) -- ($(24)!0.5!(26)$) -- ($(27)!0.5!(28)$) -- ($(25)!0.5!(29)$);
			\draw[ribbon,rs] ($(0)!0.5!(2)$) node[labelsty,rsl,left=5pt] {$r_1$} -- ($(1)!0.5!(3)$) -- ($(6)!0.5!(7)$) -- ($(8)!0.5!(12)$) -- ($(20)!0.5!(21)$) -- ($(23)!0.5!(24)$) -- ($(25)!0.5!(27)$) -- ($(28)!0.5!(29)$);
		\end{scope}
		\begin{scope}[xshift=10cm,yshift=-4cm,scale=0.8]
			\coordinate (o) at (0,0);
		  \node[gvertex,label={[labelsty]0:$r_1$}] (b1) at (1.5,0) {};
		  \node[gvertex,label={[labelsty]0:$r_2$},rotate around=45:(o)] (b2) at (b1) {};
		  \node[gvertex,label={[labelsty]45:$r_3$},rotate around=45:(o)] (b3) at (b2) {};
		  \node[gvertex,label={[labelsty]90:$r_4$},rotate around=45:(o)] (b4) at (b3) {};
		  \node[gvertex,label={[labelsty]135:$r_5$},rotate around=45:(o)] (b5) at (b4) {};
		  \node[gvertex,label={[labelsty]180:$r_6$},rotate around=45:(o)] (b6) at (b5) {};
		  \node[gvertex,label={[labelsty]225:$r_7$},rotate around=45:(o)] (b7) at (b6) {};
		  \node[gvertex,label={[labelsty]270:$r_8$},rotate around=45:(o)] (b8) at (b7) {};
		  \draw[bedge] (b1)edge(b8) (b6)edge(b7) (b6)edge(b8);
		  \draw[redge] (b2)edge(b3) (b2)edge(b4) (b2)edge(b5);
		\end{scope}
		\begin{scope}[yshift=-7.5cm]
		\begin{scope}[rotate=-90,scale=0.75]
			\node[fvertex] (0) at (0,0) {};
			\node[fvertex] (1) at (0,1) {};
			\node[fvertex,rotate around=-15:(0)] (2) at ($(0)+(1.00,0)$) {}; %-15
			\node[fvertex] (3) at ($(2)+(1)-(0)$) {};
			\node[fvertex,rotate around=45:(2)] (4) at ($(2)+(0.50,0)$) {};
			\node[fvertex] (5) at ($(4)+(3)-(2)$) {};
			\node[fvertex,rotate around=120:(1)] (6) at ($(1)+(0.50,0)$) {};
			\node[fvertex] (7) at ($(6)+(3)-(1)$) {};
			\node[fvertex,rotate around=72:(7)] (8) at ($(7)+(1.50,0)$) {};
			\node[fvertex] (9) at ($(8)+(3)-(7)$) {};
			\node[fvertex] (10) at ($(5)+(9)-(3)$) {};
			\node[fvertex] (11) at ($(8)+(10)-(9)$) {};
			\node[fvertex] (12) at ($(6)+(8)-(7)$) {};
			\node[fvertex] (13) at ($(0)+(6)-(1)$) {};
			\node[fvertex,rotate around=129:(6)] (14) at ($(6)+(1.00,0)$) {}; %-15
			\node[fvertex] (15) at ($(14)+(13)-(6)$) {};
			\node[fvertex,rotate around=93:(6)] (16) at ($(6)+(1.00,0)$) {}; %-15
			\node[fvertex] (17) at ($(16)+(14)-(6)$) {};
			\node[fvertex] (18) at ($(12)+(16)-(6)$) {};
			\node[fvertex] (19) at ($(17)+(18)-(16)$) {};
			\node[fvertex,rotate around=75:(12)] (20) at ($(12)+(0.75,0)$) {}; %-15
			\node[fvertex] (21) at ($(20)+(8)-(12)$) {};
			\node[fvertex] (22) at ($(11)+(21)-(8)$) {};
			\node[fvertex] (23) at ($(18)+(20)-(12)$) {};
			\node[fvertex] (24) at ($(23)+(21)-(20)$) {};
			\node[fvertex] (25) at ($(19)+(23)-(18)$) {};
			\node[fvertex] (26) at ($(24)+(22)-(21)$) {};
			\node[fvertex] (27) at ($(25)+(24)-(23)$) {};
			\node[fvertex] (28) at ($(27)+(26)-(24)$) {};
			\node[fvertex] (29) at ($(25)+(28)-(27)$) {};
			\draw[bedge] (0)--(2) (3)--(1) (5)--(3) (6)--(7) (12)--(8) (6)--(14) (15)--(13)
					(6)--(16) (16)--(17) (17)--(14)
					(12)--(18) (18)--(16) 
					(17)--(19) (19)--(18) 
					(12)--(20) (20)--(21) (21)--(8)
					(11)--(22) (22)--(21) 
					(18)--(23) (23)--(20) 
					(23)--(24) (24)--(21) 
					(19)--(25) (25)--(23) 
					(24)--(26) (26)--(22) 
					(25)--(27) (27)--(24) 
					(27)--(28) (28)--(26) 
					(25)--(29) (29)--(28)
					;
			\draw[redge] (0)--(1) (2)--(3) (2)--(4) (4)--(5) (5)--(3) (1)--(6)  (7)--(3)
					(7)--(8) (8)--(9) (9)--(3) (5)--(10) (10)--(9) (8)--(11) (11)--(10) 
					(6)--(12) (0)--(13) (13)--(6) (14)--(15) (18)--(16) (17)--(19)
					(22)--(21) (24)--(26) (27)--(28)  (25)--(29);
			\draw[brace,colR] (3)edge(4) (5)edge(9) (8)edge(10);
			\draw[brace,colB] (20)edge(18) (23)edge(19) (23)edge(27);
			\begin{scope}[on background layer]
				\fill[colR,fcyc] ($(2)$)--($(3)$)--($(5)$)--($(4)$)--($(2)$)--cycle;
				\fill[colR,fcyc] ($(5)$)--($(10)$)--($(9)$)--($(3)$)--($(5)$)--cycle;
				\fill[colR,fcyc] ($(3)$)--($(9)$)--($(8)$)--($(7)$)--($(3)$)--cycle;
				\fill[colR,fcyc] ($(8)$)--($(9)$)--($(10)$)--($(11)$)--($(8)$)--cycle;
				\fill[colR,fcyc] ($(0)$)--($(13)$)--($(6)$)--($(1)$)--($(0)$)--cycle;
				\fill[colB,fcyc] ($(8)$)--($(12)$)--($(20)$)--($(21)$)--($(8)$)--cycle;
				\fill[colB,fcyc] ($(20)$)--($(21)$)--($(24)$)--($(23)$)--($(20)$)--cycle;
				\fill[colB,fcyc] ($(20)$)--($(23)$)--($(18)$)--($(12)$)--($(20)$)--cycle;
				\fill[colB,fcyc] ($(18)$)--($(23)$)--($(25)$)--($(19)$)--($(18)$)--cycle;
				\fill[colB,fcyc] ($(23)$)--($(25)$)--($(27)$)--($(24)$)--($(23)$)--cycle;
				\fill[colB,fcyc] ($(6)$)--($(14)$)--($(17)$)--($(16)$)--($(6)$)--cycle;
			\end{scope}
		\end{scope}
		\begin{scope}[xshift=4.75cm,rotate=-90,scale=0.75]
			\node[fvertex] (0) at (0,0) {};
			\node[fvertex] (1) at (0,1) {};
			\node[fvertex,rotate around=0:(0)] (2) at ($(0)+(1.00,0)$) {}; %0
			\node[fvertex] (3) at ($(2)+(1)-(0)$) {};
			\node[fvertex,rotate around=45:(2)] (4) at ($(2)+(0.50,0)$) {};
			\node[fvertex] (5) at ($(4)+(3)-(2)$) {};
			\node[fvertex,rotate around=120:(1)] (6) at ($(1)+(0.50,0)$) {};
			\node[fvertex] (7) at ($(6)+(3)-(1)$) {};
			\node[fvertex,rotate around=72:(7)] (8) at ($(7)+(1.50,0)$) {};
			\node[fvertex] (9) at ($(8)+(3)-(7)$) {};
			\node[fvertex] (10) at ($(5)+(9)-(3)$) {};
			\node[fvertex] (11) at ($(8)+(10)-(9)$) {};
			\node[fvertex] (12) at ($(6)+(8)-(7)$) {};
			\node[fvertex] (13) at ($(0)+(6)-(1)$) {};
			\node[fvertex,rotate around=144:(6)] (14) at ($(6)+(1.00,0)$) {}; %+0
			\node[fvertex] (15) at ($(14)+(13)-(6)$) {};
			\node[fvertex,rotate around=108:(6)] (16) at ($(6)+(1.00,0)$) {}; %+0
			\node[fvertex] (17) at ($(16)+(14)-(6)$) {};
			\node[fvertex] (18) at ($(12)+(16)-(6)$) {};
			\node[fvertex] (19) at ($(17)+(18)-(16)$) {};
			\node[fvertex,rotate around=90:(12)] (20) at ($(12)+(0.75,0)$) {}; %+0
			\node[fvertex] (21) at ($(20)+(8)-(12)$) {};
			\node[fvertex] (22) at ($(11)+(21)-(8)$) {};
			\node[fvertex] (23) at ($(18)+(20)-(12)$) {};
			\node[fvertex] (24) at ($(23)+(21)-(20)$) {};
			\node[fvertex] (25) at ($(19)+(23)-(18)$) {};
			\node[fvertex] (26) at ($(24)+(22)-(21)$) {};
			\node[fvertex] (27) at ($(25)+(24)-(23)$) {};
			\node[fvertex] (28) at ($(27)+(26)-(24)$) {};
			\node[fvertex] (29) at ($(25)+(28)-(27)$) {};
			\draw[bedge] (0)--(2) (3)--(1) (5)--(3) (6)--(7) (12)--(8) (6)--(14) (15)--(13)
					(6)--(16) (16)--(17) (17)--(14)
					(12)--(18) (18)--(16) 
					(17)--(19) (19)--(18) 
					(12)--(20) (20)--(21) (21)--(8)
					(11)--(22) (22)--(21) 
					(18)--(23) (23)--(20) 
					(23)--(24) (24)--(21) 
					(19)--(25) (25)--(23) 
					(24)--(26) (26)--(22) 
					(25)--(27) (27)--(24) 
					(27)--(28) (28)--(26) 
					(25)--(29) (29)--(28)
					;
			\draw[redge] (0)--(1) (2)--(3) (2)--(4) (4)--(5) (5)--(3) (1)--(6)  (7)--(3)
					(7)--(8) (8)--(9) (9)--(3) (5)--(10) (10)--(9) (8)--(11) (11)--(10) 
					(6)--(12) (0)--(13) (13)--(6) (14)--(15) (18)--(16) (17)--(19)
					(22)--(21) (24)--(26) (27)--(28)  (25)--(29);
			\draw[brace,colR] (3)edge(4) (5)edge(9) (8)edge(10);
			\draw[brace,colB] (20)edge(18) (23)edge(19) (23)edge(27);
			\begin{scope}[on background layer]
				\fill[colR,fcyc] ($(2)$)--($(3)$)--($(5)$)--($(4)$)--($(2)$)--cycle;
				\fill[colR,fcyc] ($(5)$)--($(10)$)--($(9)$)--($(3)$)--($(5)$)--cycle;
				\fill[colR,fcyc] ($(3)$)--($(9)$)--($(8)$)--($(7)$)--($(3)$)--cycle;
				\fill[colR,fcyc] ($(8)$)--($(9)$)--($(10)$)--($(11)$)--($(8)$)--cycle;
				\fill[colR,fcyc] ($(0)$)--($(13)$)--($(6)$)--($(1)$)--($(0)$)--cycle;
				\fill[colB,fcyc] ($(8)$)--($(12)$)--($(20)$)--($(21)$)--($(8)$)--cycle;
				\fill[colB,fcyc] ($(20)$)--($(21)$)--($(24)$)--($(23)$)--($(20)$)--cycle;
				\fill[colB,fcyc] ($(20)$)--($(23)$)--($(18)$)--($(12)$)--($(20)$)--cycle;
				\fill[colB,fcyc] ($(18)$)--($(23)$)--($(25)$)--($(19)$)--($(18)$)--cycle;
				\fill[colB,fcyc] ($(23)$)--($(25)$)--($(27)$)--($(24)$)--($(23)$)--cycle;
				\fill[colB,fcyc] ($(6)$)--($(14)$)--($(17)$)--($(16)$)--($(6)$)--cycle;
			\end{scope}
		\end{scope}
		\begin{scope}[xshift=9.5cm,rotate=-90,scale=0.75]
			\node[fvertex] (0) at (0,0) {};
			\node[fvertex] (1) at (0,1) {};
			\node[fvertex,rotate around=15:(0)] (2) at ($(0)+(1.00,0)$) {}; %+15
			\node[fvertex] (3) at ($(2)+(1)-(0)$) {};
			\node[fvertex,rotate around=45:(2)] (4) at ($(2)+(0.50,0)$) {};
			\node[fvertex] (5) at ($(4)+(3)-(2)$) {};
			\node[fvertex,rotate around=120:(1)] (6) at ($(1)+(0.50,0)$) {};
			\node[fvertex] (7) at ($(6)+(3)-(1)$) {};
			\node[fvertex,rotate around=72:(7)] (8) at ($(7)+(1.50,0)$) {};
			\node[fvertex] (9) at ($(8)+(3)-(7)$) {};
			\node[fvertex] (10) at ($(5)+(9)-(3)$) {};
			\node[fvertex] (11) at ($(8)+(10)-(9)$) {};
			\node[fvertex] (12) at ($(6)+(8)-(7)$) {};
			\node[fvertex] (13) at ($(0)+(6)-(1)$) {};
			\node[fvertex,rotate around=159:(6)] (14) at ($(6)+(1.00,0)$) {}; %+15
			\node[fvertex] (15) at ($(14)+(13)-(6)$) {};
			\node[fvertex,rotate around=123:(6)] (16) at ($(6)+(1.00,0)$) {}; %+15
			\node[fvertex] (17) at ($(16)+(14)-(6)$) {};
			\node[fvertex] (18) at ($(12)+(16)-(6)$) {};
			\node[fvertex] (19) at ($(17)+(18)-(16)$) {};
			\node[fvertex,rotate around=105:(12)] (20) at ($(12)+(0.75,0)$) {}; %+15
			\node[fvertex] (21) at ($(20)+(8)-(12)$) {};
			\node[fvertex] (22) at ($(11)+(21)-(8)$) {};
			\node[fvertex] (23) at ($(18)+(20)-(12)$) {};
			\node[fvertex] (24) at ($(23)+(21)-(20)$) {};
			\node[fvertex] (25) at ($(19)+(23)-(18)$) {};
			\node[fvertex] (26) at ($(24)+(22)-(21)$) {};
			\node[fvertex] (27) at ($(25)+(24)-(23)$) {};
			\node[fvertex] (28) at ($(27)+(26)-(24)$) {};
			\node[fvertex] (29) at ($(25)+(28)-(27)$) {};
			\draw[bedge] (0)--(2) (3)--(1) (5)--(3) (6)--(7) (12)--(8) (6)--(14) (15)--(13)
					(6)--(16) (16)--(17) (17)--(14)
					(12)--(18) (18)--(16) 
					(17)--(19) (19)--(18) 
					(12)--(20) (20)--(21) (21)--(8)
					(11)--(22) (22)--(21) 
					(18)--(23) (23)--(20) 
					(23)--(24) (24)--(21) 
					(19)--(25) (25)--(23) 
					(24)--(26) (26)--(22) 
					(25)--(27) (27)--(24) 
					(27)--(28) (28)--(26) 
					(25)--(29) (29)--(28)
					;
			\draw[redge] (0)--(1) (2)--(3) (2)--(4) (4)--(5) (5)--(3) (1)--(6)  (7)--(3)
					(7)--(8) (8)--(9) (9)--(3) (5)--(10) (10)--(9) (8)--(11) (11)--(10) 
					(6)--(12) (0)--(13) (13)--(6) (14)--(15) (18)--(16) (17)--(19)
					(22)--(21) (24)--(26) (27)--(28)  (25)--(29);
			\draw[brace,colR] (3)edge(4) (5)edge(9) (8)edge(10);
			\draw[brace,colB] (20)edge(18) (23)edge(19) (23)edge(27);
			\begin{scope}[on background layer]
				\fill[colR,fcyc] ($(2)$)--($(3)$)--($(5)$)--($(4)$)--($(2)$)--cycle;
				\fill[colR,fcyc] ($(5)$)--($(10)$)--($(9)$)--($(3)$)--($(5)$)--cycle;
				\fill[colR,fcyc] ($(3)$)--($(9)$)--($(8)$)--($(7)$)--($(3)$)--cycle;
				\fill[colR,fcyc] ($(8)$)--($(9)$)--($(10)$)--($(11)$)--($(8)$)--cycle;
				\fill[colR,fcyc] ($(0)$)--($(13)$)--($(6)$)--($(1)$)--($(0)$)--cycle;
				\fill[colB,fcyc] ($(8)$)--($(12)$)--($(20)$)--($(21)$)--($(8)$)--cycle;
				\fill[colB,fcyc] ($(20)$)--($(21)$)--($(24)$)--($(23)$)--($(20)$)--cycle;
				\fill[colB,fcyc] ($(20)$)--($(23)$)--($(18)$)--($(12)$)--($(20)$)--cycle;
				\fill[colB,fcyc] ($(18)$)--($(23)$)--($(25)$)--($(19)$)--($(18)$)--cycle;
				\fill[colB,fcyc] ($(23)$)--($(25)$)--($(27)$)--($(24)$)--($(23)$)--cycle;
				\fill[colB,fcyc] ($(6)$)--($(14)$)--($(17)$)--($(16)$)--($(6)$)--cycle;
			\end{scope}
		\end{scope}
		\end{scope}
	\end{tikzpicture}
	\caption{Two bracings of a \Pframework\ where the first one is rigid as visible by the connectivity of the bracing graph.
	The second bracing yields a flexible framework since the bracing graph is not connected.
	We show three instances of the flex that is possible with the bracing
	and the unique resulting cartesian NAC-coloring thereof (shaded parallelograms preserves their shapes).}
	\label{fig:finalexample}
\end{figure}
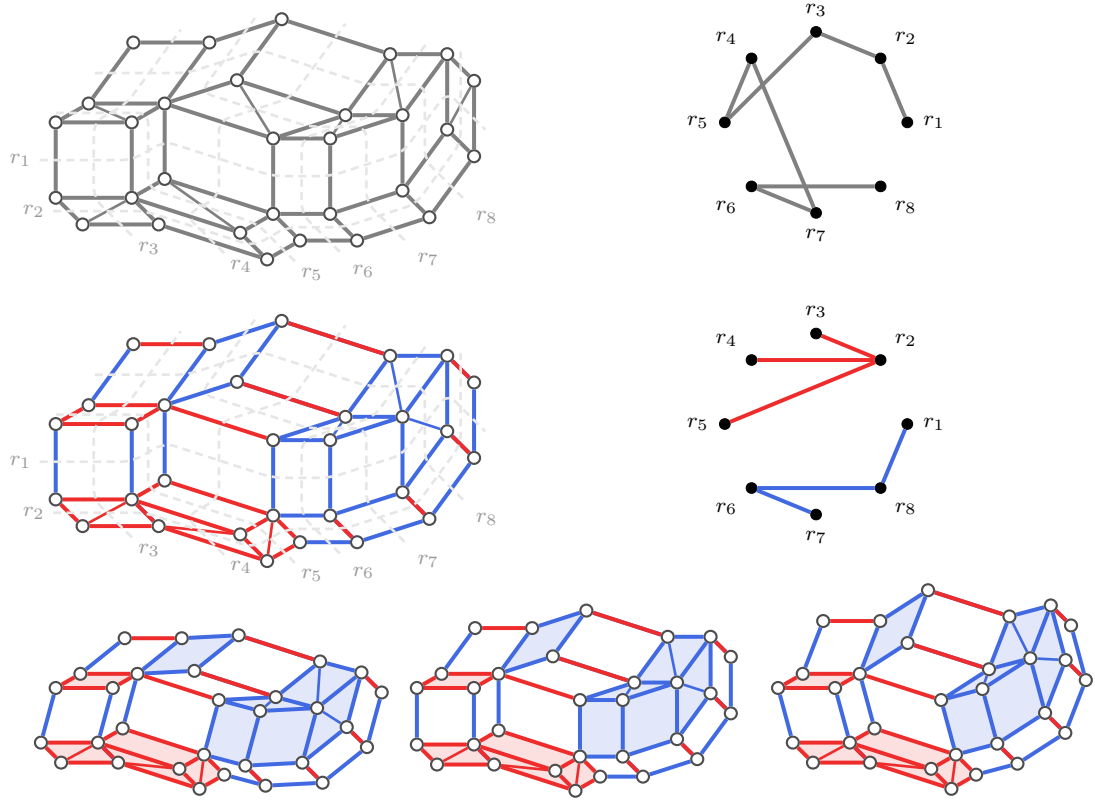

\section*{Conclusion}
We have applied the theory of NAC-colorings to \Pframework s generalizing previous results in the area of bracing grids.
In fact, we have shown that a \Pframework\ is rigid if and only if it has no cartesian NAC-coloring if and only if the bracing graph is connected.
Notice that a consequence of this statement is that a braced rectangular grid/rhombic carpet is rigid if and only if it is infinitesimally rigid.
This is not the case for grids with holes as there are instances which are rigid but not infinitesimally rigid
(an example can be obtained by bracing all squares besides those with the indicated ribbons in \Cref{fig:nonSeparatingRibbon}).

Similarly as in rectangular grids there are plenty interesting questions for further generalizations such as graphs with holes, different types of diagonals or higher dimensions.
For \Pframework s these questions are subject to further research.

\subsection*{Acknowledgments}
We thank Eliana Duarte for bringing us to this topic and sharing her ideas,
and Matteo Gallet for his comments to the paper. 
This project was supported by the Austrian Science Fund (FWF): P31061, P31888 and W1214-N15, 
and by the Ministry of Education, Youth and Sports of the Czech Republic, project no. CZ.02.1.01/0.0/0.0/16\_019/0000778.

\end{document}